\documentclass[a4paper,12pt]{amsart}
\usepackage[a4paper, margin=1in]{geometry}
\usepackage[foot]{amsaddr}

\usepackage{amsthm}
\usepackage{aliascnt}
\usepackage{amsmath,amssymb,amsfonts}
\usepackage{hyperref}
\usepackage[shortlabels]{enumitem}

\usepackage{tikz-cd}

\newcommand{\basetheorem}[3]{
	\newtheorem{#1}{#2}[#3]
	\newtheorem*{#1*}{#2}
	\expandafter\def\csname #1autorefname\endcsname{#2}
}
\newcommand{\maketheorem}[3]{
	\newaliascnt{#1}{#3}
	\newtheorem{#1}[#1]{#2}
	\aliascntresetthe{#1}
	\expandafter\def\csname #1autorefname\endcsname{#2}
	\newtheorem{#1*}{#2}
}

\basetheorem{theorem}{Theorem}{section}
\maketheorem{definition}{Definition}{theorem}
\maketheorem{lemma}{Lemma}{theorem}
\maketheorem{example}{Example}{theorem}
\maketheorem{corollary}{Corollary}{theorem}
\maketheorem{conjecture}{Conjecture}{theorem}
\maketheorem{proposition}{Proposition}{theorem}
\maketheorem{question}{Question}{theorem}
\maketheorem{remark}{Remark}{theorem}

\usepackage{etoolbox}
\AtBeginEnvironment{theorem}
{\setlength{\parskip}{0.5em}}
\AtBeginEnvironment{itemize}
{\setlength{\parskip}{0em}}
\AtBeginEnvironment{enumerate}
{\setlength{\parskip}{0em}}
\AtBeginEnvironment{reptheorem}
{\setlength{\parskip}{0.5em}}

\makeatletter
\newcommand{\newreptheorem}[2]{\newtheorem*{rep@#1}{\rep@title}\newenvironment{rep#1}[1]{\def\rep@title{#2 \ref*{##1}}\begin{rep@#1}}{\end{rep@#1}}}
\makeatother
\makeatletter
\@namedef{subjclassname@2020}{%
	\textup{2020} Mathematics Subject Classification}
\makeatother

\newreptheorem{theorem}{Theorem}

\newcommand{\D}{\Delta}
\newcommand{\Vol}{\textup{Vol}}
\newcommand{\Sing}{\textup{Sing}}

\newcommand{\nklt}{\textup{Nklt}}
\newcommand{\Ht}[1]{H^{i}_{t}}

\newcommand{\Hom}{\textup{Hom}}
\newcommand{\Fe}[1][e]{F^{#1}_{*}}

\newcommand{\ox}[1][X]{\mathcal{O}_{#1}}

\newcommand{\zp}{\mathbb{Z}_{(p)}}

\usepackage{xcolor}

\begin{document}
	\title{On the Boundedness of Globally $F$-split varieties}
	\author{Liam Stigant}
	\address{Department of Mathematics, Imperial College London, 180 Queen's Gate, 
		London SW7 2AZ, UK} 
	\email{l.stigant18@imperial.ac.uk}
	\subjclass[2020]{14J30, 14J32, 14M22 14E30, 14G17, 14B05}
	\date{\today}
	\pagestyle{myheadings} \markboth{\hfill  Liam Stigant
		\hfill}{\hfill On the Boundedness of Globally $F$-split varieties\hfill}
\begin{abstract}
	This paper proposes the use of $F$-split and globally $F$-regular conditions in the pursuit of BAB type results in positive characteristic. The main technical work comes in the form of a detailed study of  threefold Mori fibre spaces over positive dimensional bases. As a consequence we prove the main theorem, which reduces birational boundedness for a large class of varieties to the study of prime Fano varieties. 
\end{abstract}

	\maketitle

	\tableofcontents
	\setlength{\parskip}{0.5em}
	\section{Introduction}
	There has been great success proving boundedness results in characteristic zero using the techniques and results of the LMMP. Beyond dimension $2$, however, there has not been much progress in positive characteristic. This is perhaps a consequence of the relative newness of the LMMP results in this setting, but it  also tells of the existence of difficulties unique to characteristic $p$.
	
	In this direction, we prove the following.
	
	\begin{theorem}\label{Main}
		Fix $0 < \delta, \epsilon <1$. Let $S_{\delta,\epsilon}$ be the set of threefolds satisfying the following conditions
		\begin{itemize}
			\item $X$ is a projective variety over an algebraically closed field of characteristic $p >7, \frac{2}{\delta}$;
			\item $X$ is terminal, rationally chain connected and $F$-split;
			\item $(X,\Delta)$ is $\epsilon$-klt and log Calabi-Yau for some boundary $\Delta$; and
			\item The coefficients of $\Delta$ are greater than $\delta$.
		\end{itemize}
		
		Then there is a set $S'_{\delta,\epsilon}$, bounded over $\text{Spec}(\mathbb{Z})$ such that any $X\in S_{\delta,\epsilon}$ is either birational to a member of $S'_{\delta,\epsilon}$ or to some $X'\in S_{\delta,\epsilon}$, Fano with Picard number $1$. 
	\end{theorem}

	The constraints on the characteristic of the field are required to control the singularities arising in terminal Mori fibrations. In particular the $p>7$ requirement ensures that terminal del Pezzo fibrations have generically smooth fibres and the $p> \frac{2}{\delta}$ is needed to control the singularities appearing in the base of a conic bundle. This in turn allows for lower dimensional boundedness results to be applied.

	The condition that $X$ be terminal is to allow us to reduce to the case that $X$ is a terminal Mori fibre space. While we might normally achieve this by taking a terminalisation $\tilde{X} \to X$, we cannot do so while also ensuring that the coefficients of $\tilde{\D}$ are still bounded below. In fact while bounding the coefficients below is used to prove a canonical bundle formula for Mori fibre spaces of relative dimension $1$ it is in many ways the relative dimension $2$ case that forces the assumption $X$ is terminal.
		
	If $(X,\Delta) \to S$ is a klt Mori fibre space with coefficients bounded below by $\frac{2}{p}$ then we may freely take a terminalisation and run an MMP to obtain a tame conic bundle, which is what we require for our boundedness proof. If however the relative dimension is $2$ then after taking a terminalisation and running an MMP we may end with a Mori fibration of relative dimension $1$, where we cannot easily control the singularities of the base. This happens whenever $X$ is singular along a curve $C$ which maps inseparably onto the base and we expect this is the only way it might happen. 	

	The main motivation for this result comes from \cite{chen2018birational} where a similar result is proven in the characteristic zero setting. More generally we have the following generalisation of BAB, which essentially appeared in \cite{mckernan2003threefold} and remains unsolved even in characteristic $0$. 
	 
	\begin{conjecture}
		Fix $\kappa$, an algebraically closed field of characteristic $0$, let $d$ be a natural number and take $\epsilon$ a positive real number. Then the projective varieties $X$ over $\kappa$ such that 
		\begin{itemize}
			\item $X$ has dimension $d$;
			\item $(X,B)$ is $\epsilon$-klt for some boundary $B$;
			\item $-(K_{X}+B)$ is nef; and
			\item $X$ is rationally connected;
		\end{itemize}
		are bounded.	
	\end{conjecture}

	With the LMMP for klt pairs known in dimension $3$ and characteristic $p>5$, it is natural to turn our attention to results and conjectures of this type in positive and mixed characteristic. There are several major problems one would face in the pursuit of such a result, even in the weaker case of birational boundedness in dimension $3$, which do not arise in characteristic zero. Perhaps the most immediate is that $X$ rationally connected no longer removes the possibility that $K_{X} \not\equiv 0$. For example, in positive characteristic there are families of K3 surfaces which are rationally connected. It is not clear then, even in dimension $2$, that such a result would hold.
	
	It is also very difficult to control the singularities of the base, and indeed the fibres, of a Mori fibre space, which makes proofs of an inductive nature very challenging. The failure of Kawamata-Viehweg vanishing presents a similar difficulty.
	
	Unique to positive characteristic, we have singularities characterised by properties of the Frobenius morphism. In particular there are notions of globally $F$-split and globally $F$-regular which can be thought of as positive characteristic analogues of lc log Calabi-Yau varieties and klt log Fano varieties. While the exact nature of this analogy is the subject of a variety of results and conjectures, it is expected, and often known, that these varieties should behave similarly to their characteristic zero counterparts.
	
	Most notably, in this context, the $F$-split and globally $F$-regular conditions are preserved under the steps of the LMMP including Mori fibrations. In fact the conditions are also preserved under taking a general fibre of a fibration. They also come naturally equipped with vanishing theorems, with globally $F$-regular pairs satisfying full Kawamata-Viehweg vanishing. 
	
	We also have some relevant characterisations of uniruled $F$-split varieties. If $X$ is smooth it cannot be simultaneously $F$-split, Calabi-Yau and uniruled. In particular, an $F$-split, canonical surface cannot be uniruled and have pseudo-effective canonical divisor. 
	
	In many ways then, global $F$-singularities begin to resolve the most obvious difficulties in proving positive characteristic boundedness results. They present their own problems however, there is no satisfactory notion of ``$\epsilon$-$F$-split'' or ``$\epsilon$-globally $F$-regular'' which makes it difficult to work solely with these notions in the context of boundedness.
	
	That said, while the $F$-split and globally $F$-regular conditions fit naturally into the study of log pairs, we may also choose to consider them as properties of the underlying base varieties. In such a way we may formulate the following questions, though in practice even the most optimistic might expect further conditions on the characteristic. One could also reasonably ask that the $\epsilon$-klt pair $(X,B)$ is itself $F$-split, or globally $F$-regular, in place of the base variety.
	
	\begin{question}\label{Q1}
		Fix $d$ a natural number and $\epsilon$ a positive real number. Then is the set, $S$, (resp. $S'$) of projective varieties $X$ such that $(1)-(4)$ (resp. $(1),(2),(3'),(4')$) hold bounded over $\mathbb{Z}$?
		\begin{enumerate}[label=(\arabic*)]
			\item $X$ has dimension $d$ over some closed field $\kappa$.
			\item $(X,B)$ is $\epsilon$-klt for some boundary $B$.
			\item $-(K_{X}+B)$ is big and nef.
			\item[$(3')$] $K_{X}+B\equiv 0$.
			\item If $\kappa$ has characteristic $p>0$, then $X$ is globally $F$-regular.
			\item[$(4')$] If $\kappa$ has characteristic $p>0$, then $X$ is $F$-split and rationally chain connected.
		\end{enumerate}
	\end{question}

	\begin{remark}
		Here rationally chain connected is chosen over rationally connected in light of \cite{gongyo2015rational} which shows that globally $F$-regular threefolds are rationally chain connected in characteristic $p > 7$. Further in characteristic zero, under mild assumptions on the singularities (X admits a boundary $\D$ with $(X,\D)$ dlt), rational chain connectedness coincides with rational connectedness so this is still a natural generalisation. In any case, in dimension $3$ the globally $F$-regular condition is strictly stronger than $F$-split and rationally chain connected whenever the characteristic is greater than $7$. 
		
		In fact other than the case of Fano varieties of Picard number $1$, Gongyo et al are able to show separable rational connectedness. This might, therefore, also be a natural condition to impose instead, especially since the classical proof of the boundedness of characteristic zero prime Fano threefolds so heavily relies on the existence of a free curve. 
	\end{remark}
	
	Given \autoref{Q1}, it is natural to ask what can be gleaned from \autoref{Main} about globally $F$-regular varieties of the type described in \autoref{Q1}. Unfortunately the answer is very little, while every globally $F$-regular variety is $F$-split and if $X$ is of $\epsilon$-log Fano type it is also of $\epsilon$-LCY type, we cannot sensibly ensure that the resulting $\epsilon$-LCY pair $(X,\Delta)$ has coefficients bounded below, even if we require it for the pair $\epsilon$-log Fano pair $(X,\Delta')$.\\
	
	As part of this work we prove the following weak BAB result in \autoref{J1} and \autoref{J2}. This draws heavily on the arguments of Jiang in \cite{jiang2014boundedness}.
	
	\begin{theorem}\label{Main2}
		Fix $0 < \delta, \epsilon <1$ and let $T_{\delta,\epsilon}$ be the set of threefold pairs $(X,\Delta)$ satisfying the following conditions
		\begin{itemize}
			\item $X$ is projective over a closed field of characteristic $p >7,\frac{2}{\delta}$;
			 \item $X$ is terminal, rationally chain connected and $F$-split;
			\item $(X,\Delta)$ is $\epsilon$-klt and LCY;
			\item The coefficients of $\Delta$ are greater than $\delta$; and
			\item $X$ admits a Mori fibre space structure $X \to Z$ where $Z$ is not a point.
		\end{itemize}
	Then the set $\{\Vol(-K_{X}) \colon \exists \Delta \text{ with } (X,\Delta) \in T_{\delta,\epsilon}\}$ is bounded above.
	\end{theorem}
	\begin{remark}
		Together with the observation that taking a terminalisation and running a $K_{X}$-MMP can only increase the anti-canonical volume, we reduced weak BAB for varieties in $S_{\D,\epsilon}$ to the case of prime Fano varieties of $\epsilon$-LCY type. Over a fixed field, however, this is essentially superseded by the result of \cite{das2018boundedness}, which gives weak BAB for varieties $X$ with $K_{X}+\D \equiv 0$ for some boundary $\D$ taking coefficients in a DCC set and making $(X,\D)$ klt.
		
		Results similar to \autoref{Main} and \autoref{Main2} are proven in \cite[Theorem 1.7, Theorem 1.8]{ZhuangFano} for Fano threefolds satisfying certain conditions on the Seshadri constant at a smooth closed point. Further these conditions are closely related to global $F$-regularity by \cite[Theorem 1.3]{ZhuangFano}.
	\end{remark}

	We begin by collecting some relevant definitions and results for later usage. Then \autoref{S-CB} establishes key results about the behaviour of conic bundles in sufficiently high characteristic. Next \autoref{S-MFS} contains the key boundedness arguments, with weak BAB deferred to \autoref{S-BAB}. Finally \autoref{Main} is proved in \autoref{S-res}.
	
\section{Definitions}
	
\subsection{MMP Singularities}
Here $\mathbb{K}$ will be taken to mean either $\mathbb{R}$ or $\mathbb{Q}$. If no field is specified, it is taken to be $\mathbb{R}$. We outline the key notions of singularity arising in the MMP.

\begin{definition}
	Let $X$ be a normal variety. A \emph{$\mathbb{K}$-boundary} is an effective $\mathbb{K}$-divisor $\Delta$ where $K_{X}+\Delta$ is $\mathbb{K}$-Cartier and the coefficients of $\Delta$ are at most $1$. 
	
	A \emph{$\mathbb{K}$ pair} is a couple $(X,B)$ where $X$ is normal and $B$ is a $\mathbb{K}$-boundary.
	
	If $B$ is not effective but $(X,B)$ would otherwise be a $\mathbb{K}$ pair we call it a \emph{$\mathbb{K}$ sub pair}.
\end{definition}

Since $K_{X}+\Delta$ is $\mathbb{R}$-Cartier, we may pull it back along any morphism $\pi\colon  Y \to X$. If $\pi$ is birational then there is a unique choice of $\Delta_{Y}=\sum -a(Y,E,X,\Delta)E$ which agrees with $\Delta$ away from the exceptional locus of $\pi$ such that $\pi^{*}(K_{X}+\Delta)\sim_{\mathbb{R}}K_{Y}+\Delta_{Y}$. In a slight abuse of notation we write $f^{*}(K_{X}+\Delta)=(K_{Y}+ \Delta_{Y})$.

Suppose that $f\colon Y \to X$ is birational morphism of normal varieties and there is a some normal variety $Z$ with $g\colon Z\to Y$. If $E$ is a divisor on $Y$ with strict transform $E'$ on $Z$ then $a(Z,E',X,\Delta)=a(Z,E',Y,\Delta_{Y})=a(Y,E,X,\Delta)$. We may view then the coefficients $a(Y,E,X,\Delta)$ as being independent of $Y$ and write $a(E,X,\Delta)$ instead. 

\begin{definition}
	Given a sub pair $(X,\Delta)$ we define the \emph{discrepancy} $$\text{Disc}(X,\Delta):=\inf \{a(E,X,\Delta) \text{ such that } E \text{ is exceptional and has non-empty center on } X\}$$
	and the \emph{total discrepancy} 
	
	$$\text{TDisc}(X,\Delta):=\inf \{a(E,X,\Delta) \text { such that } E \text{ has non-empty center on } X\}$$
\end{definition}

We then use these to define a suite of singularities.

\begin{definition}
	Let $(X,\Delta)$ be a (sub) pair then we say that $(X,\Delta)$ is
	\begin{itemize}
		\item 	\emph{(Sub) terminal} if $\text{Disc}(X,\Delta) > 0$.
		\item	\emph{(Sub) canonical} if $\text{Disc}(X,\Delta)\geq 0$.
		\item   \emph{(Sub) plt} if $\text{Disc}(X,\Delta) > -1$.
		\item	\emph{(Sub) $\epsilon$-klt} if $\text{TDisc}(X,\Delta) > \epsilon-1$.
		\item	\emph{(Sub) $\epsilon$-lc} if $\text{TDisc}(X,\Delta) \geq \epsilon -1$.
	\end{itemize}
\end{definition}
For $\epsilon=0$ we say klt, lc respectively.

When we have resolution of singularities there is another, more practical version.

\begin{definition}
	Let $(X,\Delta)$ be a (sub) pair and $\pi\colon Y\to X$ a log resolution of $(X,\Delta)$. Let $$t=\min\{a(E,X,\Delta) \text{ such that } E \text{ is a divisor on } Y\}$$ and $$d=\min\{a(E,X,\Delta) \text{ such that } E \text{ is an exceptional divisor of } \pi\colon Y\to X\}.$$
	Then $(X,\Delta)$ is
	\begin{itemize}
		\item	\emph{(Sub) $\epsilon$-klt} if $t > \epsilon-1$;
		\item	\emph{(Sub) $\epsilon$-lc} if $t\geq \epsilon -1$.
	\end{itemize}
	If $\Delta=0$ then $X$ is 
	\begin{itemize}
		\item 	\emph{terminal} if $d > 0$;
		\item	\emph{canonical} if $d\geq 0$;
		\end{itemize}
\end{definition}

This also gives rise to an additional notion of singularity, which is dependent on the choice of resolution and can be thought of as the limit of a klt pair.

\begin{definition}
	A pair $(X,\Delta)$ is called \emph{dlt} if there is a log resolution $\pi\colon Y \to X$ of $(X,\Delta)$ with $K_{Y}+\Delta_{Y}=\pi^{*}(K_{X}+\Delta)$ such that $\text{Coeff}_{E}(\Delta_{Y}) < 1$ for every $E$ exceptional.  
\end{definition}

If $(X,\Delta)$ is sub klt or sub lc etc and $\pi\colon Y \to X$ is a birational morphism from a normal variety and $K_{Y}+\Delta_{Y}=\pi^{*}(K_{X}+\Delta)$ then $(Y,\Delta_{Y})$ has the same singularities. Conversely we have the following.

\begin{lemma}\cite[Lemma 3.38]{kollar2008birational}\label{comparisonLemma}
	Suppose $(X,\Delta),(X',\Delta')$ are pairs equipped with proper birational morphisms $f\colon X \to Y$ and $f'\colon X'\to Y$ with $f_{*}\Delta=f'_{*}\Delta'$.
	
	Suppose further that $-(K_{X}+\Delta)$ is $f$ nef and $(K_{X'}+\Delta')$ is $f'$ nef. Then $a(E,X,\Delta) \leq a(E,X',\Delta')$ for any $E$ with non-trivial center on $Y$.
\end{lemma}

In particular, these notions of singularity are preserved under a $(K_{X}+\Delta)$ MMP. 

\begin{definition}
	A (sub) $\epsilon$-klt pair $(X,\Delta)$ where $K_{X}+\Delta \equiv 0$ is said to be \emph{(sub) $\epsilon$-log Calabi-Yau}, or just (sub) $\epsilon$-LCY. 
	
	If instead $-(K_{X}+\Delta)$ is big and nef, it is said to be \emph{(sub) $\epsilon$-log Fano}. 
\end{definition}
Again for $\epsilon=0$ we just say LCY and log Fano, equally if $\Delta=0$ we drop the log.

Of particular interest is the class of prime Fano varieties which we may think of as Mori fibre spaces over a point.
\begin{definition}
	A terminal Fano variety is said to be \emph{prime} if it has Picard rank $1$.
\end{definition}

\begin{corollary}
	Suppose that $(X,\Delta)$ is (sub) $\epsilon$-LCY and $f\colon X\dashrightarrow X'$ is either a flip or a divisorial contraction then $(X',f_{*}\Delta)$ is (sub) $\epsilon$-LCY.
\end{corollary}
\begin{proof}
	Both $(K_{X}+\Delta)$ and $(K_{X'}+\Delta')$ are numerically trivial so it suffices to show that $(K_{X'}+\Delta')$ is $\mathbb{R}$-Cartier by \autoref{comparisonLemma}. 
	
	If $g\colon X \to Y$ is the contraction of an extremal ray and $D\equiv_{g} 0$ is Cartier, there is some $L$ Cartier on $Y$ with $g^{*}L=D$. 
	Suppose first that $f$ is a divisorial contraction. Then $K_{X}+\Delta=f^{*}L$, say, and so $K_{X'}+\Delta'=L$ by the projection formula. 
	
	Otherwise $f$ is a flip and there is $g\colon X \to Y$ a flipping contraction together with $g'\colon X' \to Y$ such that $f=g'^{-1}\circ g$. Hence writing $K_{X}+\Delta=g^{*}L$ again gives $K_{X'}+\Delta'=g'^{*}L$.
	
	In either case, $(K_{X'}+\Delta')$ is $\mathbb{R}$-Cartier.
\end{proof}

We will be interested in LCY varieties in which general points can be connected by rational curves in the following senses.

\begin{definition}
	Let $X$ be a variety over a field $\kappa$. Then $X$ is said to be:
	\begin{itemize}
		\item \emph{Uniruled} if there is a proper family of connected curves $f\colon U \to Y$ where the generic fibres have only rational components together with a dominant morphism $U \to X$ which does not factor through $Y$.
		\item \emph{Rationally chain connected (RCC)} if there is $f\colon U \to Y$ as above such that $u^{2}\colon U \times_{Y} U \to X \times_{k} X$ is dominant.
		\item \emph{Rationally connected} if there is $f\colon U \to Y$ as above witnessing rational chain connectedness such that the general fibres are irreducible.
		\item \emph{Separably rationally connected} if $f$ as above is separable.
	\end{itemize}
\end{definition}

If $X \to X'$ is a dominant morphism from $X$ uniruled/RCC/rationally connected then we may compose $U \to X \to X'$ to see that $X'$ is uniruled/RCC/rationally connected.

\subsection{$F$-Singularities of Pairs}

We now introduce Frobenius singularities, unique to positive characteristic. We focus on the $F$-pure and $F$-split conditions, as $F$-regularity will not be needed.

\begin{definition}
	Given a $\kappa$ algebra $R$ in positive characteristic we denote the Frobenius morphism by $F\colon R\to R$ sending $x \to x^{p}$. Any $R$ module $M$ then has an induced module structure, denoted $F_{*}M$ where $R$ acts as $r.x=F(r)x=r^{p}x$. Finally $R$ is said to be \emph{$F$-finite} if $F_{*}R$ is a finite $R$ module.
	These definitions naturally extend to schemes over $\kappa$. 
\end{definition}

Note that all perfect fields are $F$-finite, and so is every variety over an $F$-finite field.
In this context we can view the Frobenius morphism as a map of $R$ modules $F\colon R \to F_{*}R$. We will also write $F^{e}\colon R \to F_{*}^{e}R$ for the $e^{th}$ iterated Frobenius.
\begin{definition}
	Let $X$ be a variety over an $F$-finite field.
	We say $X$ is:
	\begin{itemize}
		\item \emph{$F$-pure} if the Frobenius morphism $\ox \to F_{*}\ox$ is pure, or equivalently locally split.
		\item \emph{(Globally) $F$-split} if the Frobenius morphism $\ox \to F_{*}\ox$ is split.
	\end{itemize} 
\end{definition}

Being $F$-split is a particularly strong condition, giving the following vanishing result almost immediately.

\begin{lemma}\label{vanish}
	Let $X$ be an $F$-split variety and $A$ an ample $\mathbb{Q}$-Cartier divisor. Then $H^{i}(X,A)=0$ for all $i>0$. 
\end{lemma} 

\begin{proof}
	By assumption $\ox \to \Fe \ox $ splits, and hence so does $A=A\otimes \ox \to A \otimes \Fe \ox =\Fe A^{p^{e}}$. That is we have $id\colon  A \to \Fe A^{p^{e}} \to A$, and taking cohomology we see that $H^{i}(X,A)$ injects into $H^{i}(X,\Fe A^{p^{e}})=H^{i}(X,A^{p^{e}})$ which vanishes for $e>>0$. 
\end{proof}
Take $X$ a normal variety. To mirror the notion of a boundary we introduce pairs $(\mathcal{L}, \phi)$ where $\mathcal{L}$ is a line bundle and $\phi\colon  \Fe \mathcal{L} \to \ox$. By applying duality on the smooth locus, which contains all the codimension $1$ points we observe that $\Hom_{\ox}(\Fe \mathcal{L},\ox)=H^{0}(X,\mathcal{L}^{-1}((1-p^{e})K_{X}))$. Therefore such a pair corresponds to a divisor $\Delta_{\phi} \geq 0$ with $(p^{e}-1)(K_{X}+\Delta_{\phi}) \sim \mathcal{L}$. \\ Reversing this procedure is slightly more involved. If ${(p^{e}-1)(K_{X}+\Delta) \sim \mathcal{L}}$ (we write this $K_{X}+\Delta \sim_{\zp} \mathcal{L}$) we may obtain $\phi_{\Delta}\colon \Fe \mathcal{L} \to \ox$, however we could also write say $(p^{2e}-1)(K_{X}+\Delta) \sim \mathcal{L'}$ where $\mathcal{L'} \not\sim \mathcal{L}$. We introduce, therefore, the following notion of equivalence.

First, we say that two such pairs, $(\mathcal{L}, \phi)$ and $(\mathcal{L}', \phi')$ are equivalent if:
\begin{itemize}
	\item There is an isomorphism $\psi: \mathcal{L} \to \mathcal{L'}$ such that following diagram commutes; or
	\[\begin{tikzcd}
	\Fe \mathcal{L} \arrow[rd, "\phi"] \arrow[rr, "\Fe \psi"] &     & \Fe \mathcal{L}' \arrow[ld, "\phi'"] \\
	& \ox &                                     
	\end{tikzcd}\]
	\item $\mathcal{L}=\mathcal{L}^{p^{e'}+1}$ and $\phi'\colon \Fe[e+e']\mathcal{L}^{p^{e'}+1}\to \ox$ is the precisely the map given by
	$$\Fe[e+e'](\mathcal{L}\otimes \mathcal{L}^{p^{e'}}) \xrightarrow{\Fe\phi} \Fe \mathcal{L} \xrightarrow{\phi} \ox.$$
	
\end{itemize}

We then expand the notion of equivalence to allow any finite combination of the above equivalences, more precisely we take the transitive closure of our initial relation.
This gives a bijection between equivalence classes of pairs $(\mathcal{L}, \phi)$ and $\Delta \geq 0$ with $(K_{X}+\Delta)$ $\zp$-Cartier. Full details on such pairs can be found in Chapter 16 of Schwede's notes on $F$-singularities \cite{schwede2010f}. 

To extend this framework to allow for sub pairs we can instead work with morphisms $\Fe\mathcal{L} \to K(X)$ where we view $K(X)$ as a constant sheaf on $X$. Given such a morphism $\phi$, we can always find $E \geq 0$ Cartier such that when we twist by $E$ we obtain $\phi'\colon =\Fe(\mathcal{L}((1-p^{e})E)) \to \ox$ and thus associate a divisor $\Delta_{\phi'}$ with $(1-p^{e})(K_{X}+\Delta_{\phi'})\sim \mathcal{L}((1-p^{e})E$ and then take $\Delta_{\phi}=\Delta_{\phi'}-E$.

\begin{lemma}\cite[Lemma 2.3]{das2015f}
	With the notation as above, $\Delta_{\phi}$ does not depend on the choice of $E$.
\end{lemma}

\begin{definition}
	A sub $\zp$ pair is a couple $(X,B)$ where $(K_{X}+B)$ is $\zp$-Cartier and the coefficients of $B$ are at most $1$. We write $\phi{e}_{B}\colon  F_{*}^{e,B}\mathcal{L}_{e,B} \to K(X)$ for the associated morphism dropping the dependence on $B$ when it remains clear. If $B$ is effective $(X,B)$ is called a $\zp$ pair and we view $\phi$ as being a morphism to $\ox$.
	
	Let $(X,B)$ be a (sub) $\zp$ pair, then $(X,B)$ is
	\begin{itemize}
		\item \emph{(sub) $F$-pure} if $\ox \subseteq \text{Im}(\phi^{e})$ for some $e$.
		\item \emph{(sub) $F$-split} if $1\in\text{Im}(H^{0}(X,\phi^{e}))$ for some $e$.
	\end{itemize}
\end{definition}
Being $F$-split is also sometimes called globally $F$-split to distinguish to from $F$-pure, which can be thought of as being locally split.

Locally to a point of codimension $1$ these definitions are particularly well-behaved.

\begin{lemma}\cite[Lemma 2.14]{das2015f}
	Let $R$ be a regular DVR with parameter $t$, then a sub $\zp$ pair $(R,\lambda t)$ is sub $F$-pure iff $\lambda \leq 1$ and sub $F$-regular iff $\lambda < 1$.
\end{lemma}

In particular we see that the coefficient of $\Delta_{\phi}$ at $E$ depends only on $\phi$ near $E$.
\begin{corollary}\label{local}
	Suppose $\phi\colon \Fe\mathcal{L} \to K(X)$ has associated divisor $\Delta$ then $1-\text{Coeff}_{E}(\Delta)=\inf\{t\colon (X,\Delta+tE) \text{ is } F \text{-pure at the generic point of } E\}$. 
\end{corollary}

While these definitions do not pullback along birational morphisms as obviously as the usual MMP singularities, it is still possible.

\begin{lemma}\cite[Lemma 7.2.1]{blickle2013p}
	Suppose that $f\colon X \to Y$ is a birational morphism with $X$ normal and $(Y,\Delta)$ a sub $F$-split pair. Then there is $\Delta'$ on $X$ making $(X,\Delta')$ a sub $F$-split pair such that $(K_{X}+\Delta')=f^{*}(K_{Y}+\Delta)$.  
\end{lemma}
\begin{proof}
	Take the corresponding map $\phi\colon  \Fe\mathcal{L} \to K(Y)$. Then we may freely view $\mathcal{L}$ as a subsheaf of $K(X)$ and so extend $\phi$ to a map $\phi\colon  \Fe K(Y) \to K(Y)$. Taking the inverse image gives $f^{-1}(\phi)\colon f^{-1}\Fe K(Y) \to f^{-1}K(Y)$ and $f^{-1}\Fe \mathcal{L} \to f^{-1}K(Y)$. Since $\pi$ is birational we obtain an isomorphism $f^{-1}K(Y) \to K(X)$. We then have the following situation.
	\[\begin{tikzcd}
	f^{-1}\Fe(\mathcal{L}) \otimes_{f^{-1}\Fe\mathcal{O}_{Y}}\Fe\ox \arrow[r, hook] & \Fe K(X) \arrow[r]                                          & K(X)                                \\
	f^{-1}\Fe(\mathcal{L}) \arrow[r, hook] \arrow[u, hook]                       & f^{-1}\Fe K(Y) \arrow[r, "f^{-1}(\phi)"'] \arrow[u, "\sim"] & f^{-1}K(Y) \arrow[u, "\sim", hook']
	\end{tikzcd}\]
	
	Note however that $f^{-1}\Fe(\mathcal{L}) \otimes_{f^{-1}\Fe\mathcal{O}_{Y}}\ox= \Fe f^{*}\mathcal{L}$ and hence we obtain the desired map $\tilde{\phi}\colon \Fe f^{*}\mathcal{L} \to K(X)$. This induces a divisor $\Delta'$ on $X$ with $$(p^{e}-1)(K_{X}+\Delta') \sim f^{*}\mathcal{L} \sim (p^{e}-1)f^{*}(K_{Y}+\Delta).$$ The coefficient of $\Delta'$ at a codimension one point can be recovered from $\tilde{\phi}$ by working locally around that point. In particular, wherever $f$ is an isomorphism, $\phi$ and $\tilde{\phi}$ agree. Therefore the coefficients of $\Delta$ and $\Delta'$ agree on this locus also, so we have $f^{*}(K_{Y}+\Delta)=(K_{X}+\Delta')$ as required. Moreover commutativity of the earlier diagram gives that whenever $1 \in \text{Im}(H^{0}(Y,\phi))$ then it is also in the image of $H^{0}(X,\tilde{\phi})$, and hence $(X,\Delta)$ is sub $F$-split.
\end{proof}

In general the local forms of these singularities cannot be pushed forward, however the global ones often can be, even along morphisms which are not birational.

\begin{lemma}\cite[Theorem 5.2]{das2015f}
	Suppose that $(X,\Delta)$ is sub $F$-split and there is a map $f\colon X \to Y$ with $f_{*}\ox =\mathcal{O}_{Y}$ and $(K_{X}+\Delta) \sim_{\zp} f^{*}\mathcal{L}$ for some line bundle $\mathcal{L}$ on $Y$. If every component of $\Delta$ which dominates $Y$ is effective then there is $\Delta_{Y}$ with $(Y,\Delta_{Y})$ sub $F$-split and $\mathcal{L}\sim_{\zp} (K_{Y}+\Delta_{Y})$. 
\end{lemma}

If $f\colon X \to Y$ is birational then the conditions are automatically satisfied and the induced $\Delta_{Y}$ is just the pushforward $f_{*}\Delta$ by \autoref{local}. Therefore if $X$ is sub $F$-split so is every $X'$ birational to $X$. Further if $X$ is $F$-split and $X'$ is obtained by taking a terminalisation or running a $K_{X}+B$ MMP for any $B$ then $X'$ is $F$-split.

\subsection{Boundedness}

Finally we introduce the relevant notions of boundedness.

\begin{definition}\label{d_birationally-bounded} We say that a set $\mathfrak{X}$ 
	of varieties is \emph{birationally bounded over a base $S$} if there is a flat, projective family $Z \to T$, where $T$ is a reduced quasi-projective scheme over $S$, such that every $X\in \mathfrak{X}$ is birational to some geometric fibre of $Z \to T$. If the base is clear from context, say if every $X \in \mathfrak{X}$ has the same base, we omit dependence on $S$.
	
	If for each $X \in \mathfrak{X}$ the map to a geometric fibre is an isomorphism we say that $\mathfrak{X}$ is \emph{bounded over $S$}.
\end{definition}

If $S=\text{Spec }{R}$ we often just say (birationally) bounded over $R$. In practice we characterise boundedness over $\mathbb{Z}$ via the following result, coming from existence of the Hilbert and Chow schemes.

\begin{lemma}\cite[Proposition 5.3]{tanaka2019boundedness}
	Fix integers $d$ and $r$. Then there is a flat projective family $Z \to T$ where $T$ is a reduced quasi-projective scheme over $\mathbb{Z}$ satisfying the following property. If
	\begin{enumerate}
		\item $\kappa$ is a field;
		\item $X$ is a geometrically integral projective scheme of dimension $r$ over $\kappa$; and
		\item there is a closed immersion $j\colon X \to \mathbb{P}^{m}_{\kappa}$ for some $m\in \mathbb{Z}$ such that $j^{*}(\mathcal{O}(1))^{r} \leq d$.
	\end{enumerate}

	Then $X$ is realised as a geometric fibre of $Z \to T$
\end{lemma}

\begin{corollary}\label{l_birationally-bounded}
	Suppose $\mathfrak{X}$ is a set of varieties over closed fields and there are positive real numbers $d,V$ such that for every $X \in \mathfrak{X}$,
	\begin{itemize}
		\item $X$ has dimension at most $d$; and
		\item There is $M$ on $X$ with $\phi_{|M|}$ birational and $\Vol(M)\leq V$.
	\end{itemize}
Then $\mathfrak{X}$ is birationally bounded over $\mathbb{Z}$. If in fact each $M$ is very ample then $\mathfrak{X}$ is bounded. 
\end{corollary}

	Conversely, if $S$ is Noetherian then we may always choose $H$ relatively very ample on $Z \to T$ with trivial higher direct images. The restriction of $H$ to any geometric fibre is therefore very ample, and of bounded degree. 
	\section{Preliminary Results}

	In this section we gather necessary results for later usage. We begin with some results on surfaces, followed by some MMP results and their applications. We also collect some Bertini type theorems at the end of the section.

	\begin{theorem}\cite[Theorem 6.9]{alexeev1994boundedness}\label{BAB}
		Fix $\epsilon >0$ and an algebraically closed field of arbitrary characteristic. Let $S$ be the set of all projective surfaces $X$ which admit a $\Delta$ such that:
		\begin{itemize}
			\item $(X,\Delta)$ is $\epsilon$-klt;
			\item $-(K_{X}+\Delta)$ is nef; and
			\item Any of the following holds $K_{X} \not\equiv 0$, $\Delta \neq 0$, $X$ has worse than Du Val singularities.
		\end{itemize}
		Then $S$ is bounded.
	\end{theorem}

	Alexeev shows boundedness over a fixed field, however it is not immediately clear if such varieties are collectively bounded over $\mathbb{Z}$. We briefly show that his methods can be extended, via the arguments of \cite{witaszek2015effective} to give a boundedness result in mixed characteristic.
	
	\begin{theorem}\label{SBAB}
			Fix $\epsilon$ a positive real number. Let $S$ be the set of projective surfaces $X$ such that following conditions hold:
		\begin{itemize}
			\item $X$ is a variety over some closed field $\kappa$;
			\item $(X,B)$ is $\epsilon$-klt for some boundary $B$;
			\item $-(K_{X}+B)$ is nef; and
			\item $X$ is rationally chain connected and $F$-split (if $\kappa$ has characteristic $p$).
		\end{itemize}
		Then $S$ is bounded.
	\end{theorem}
	\begin{proof}
		We consider first $\hat{S}:=\{X \in S\colon  K_{X} \not\equiv 0\}$. Take any such $X \in \hat{S}$, then by Alexeev \cite[Chapter 6]{alexeev1994boundedness} we have the following:
		\begin{itemize}
			\item The minimal resolution $\tilde{X}\to X$ has $\rho(X) < A $, for some constant A, depending only on $\epsilon$ and admits a birational morphism to $\mathbb{P}^{2}$ or $\mathbb{F}_{n}$ for $n < \frac{2}{\epsilon}$. In particular there is a set $T_{\epsilon}$ bounded over $\mathbb{Z}$ such that every $\tilde{X}$ is a blowup of some $Y \in T_{\epsilon}$ along a finite length subscheme of dimension $0$. That is the set of minimal desingularisations is bounded over $\mathbb{Z}$.
			\item We may run a $K_{X}$-MMP to obtain $X'$ a Mori fibre space. 
			\item There is an $N$, independent of the field of definition, such that $NK_{X'}$ is Cartier for any Mori fibre space $X'$ obtained as above.
			\item $\Vol(-K_{X'})$ is bounded independently of the base field.
			\item If $X'$ is such a Mori fibre space $X' \to \mathbb{P}^{1}$ and $F$ a general fibre then $-K_{X} +(\frac{2}{\epsilon}-1)F$ is ample.
		\end{itemize}
		
		It is sufficient then to show $S'=\{X' \text{ an } \epsilon-\text{LCY type, Mori fibre space }\}$ is bounded in mixed characteristic, then $\hat{S}$ is bounded by sandwiching as in Alexeev's original proof and the full result follows. In turn by \autoref{l_birationally-bounded} it is enough to find $V$ such that every $X' \in S'$ has a very ample divisor, $H$, satisfying $H^{2}\leq V$. We do this first for positive characteristic varieties.
		
		Fix, then, $m > \frac{2}{\epsilon}-1$ and suppose $X'\to \mathbb{P}^{1}$ is a Mori fibre space in positive characteristic. Then $A=-K_{X'} +mF$ is ample and $NA$ is Cartier. Further we have that $A'=7NK_{X'}+27N^{2}A=(7N-27N^{2})K_{X'}+27N^{2}mF$ is very ample by \cite[Theorem 4.1]{witaszek2015effective}. Since $F$ is base point free, we may add further multiples of $F$ and consider the very ample Cartier divisor $\hat{A}=(27N^{2}-7N)(-K_{X'}+2mF)$. Then $$\hat{A}^{2}=\Vol(X',\hat{A})\leq (27N^{2}-7N^{2})(\Vol(X',-K_{X'})+2m\Vol(F,-K_{F}))$$ which is bounded above, since $\Vol(X',-K_{X'})$ is bounded and $\Vol(F,-K_{F})=2$. 
		
		Similarly if $X'$ has $\rho(X')=1$ and $-K_{X'}$ ample then $-nK_{X'}$ is a very ample Cartier divisor with vanishing higher cohomology for some $n$ fixed independently of $X'$. Then $(-nK_{X'})^{2}=n^{2}\Vol(X,-K_{X'})$ is bounded and the result follows similarly.
		
		Suppose then that $X \in S$ with $K_{X} \equiv 0$, then it must have worse than canonical singularities by \autoref{split}. Let $\pi\colon Y \to X$ be a minimal resolution, with $K_{Y}+B=\pi^{*}K_{X} \equiv 0$ and $B >0$, then $Y$ is still $\epsilon$-klt, so $Y \in \hat{S}$. Consequently $X$ has $\mathbb{Q}$-Cartier Index dividing $N$ also. Moreover, there is $H$ on $Y$ very ample with $H^{2}$ bounded above. Let $H'=\pi_{*}H$, so that $NH'$ is ample and Cartier on $X$. Applying \cite[Theorem 4.1]{witaszek2015effective} again we see that $A\equiv 27N^{2}H$ is very ample, since $K_{X}\equiv 0$, with $A^{2}$ bounded above.
		
		The arguments in characteristic $0$ are essentially the same, making use of Koll{\'a}r's effective base-point freeness result \cite[Theorem 1.1, Lemma 1.2]{kollar1993effective} instead of Witaszek's result, and the existence of very free rational curves on smooth rationally connected surfaces instead of \autoref{split}.
	\end{proof}

\begin{remark}
	In particular we have an affirmative answer to Question 1 in dimension $2$.
\end{remark}

\begin{theorem}\cite[Theorem 1.2]{patakfalvi2019ordinary}\label{WO-unir}
	Let $X$ be a normal, Cohen Macaulay variety with $W\mathcal{O}$-rational singularities over a perfect field of positive characteristic. Then $X$ cannot simultaneously satisfy all the following conditions.
	\begin{enumerate}
		\item $X$ is uniruled.
		\item $X$ is $F$-split.
		\item $X$ has trivial canonical bundle.
	\end{enumerate}
	If in fact $X$ is smooth then we may replace $K_{X}\sim 0$ with $K_{X} \equiv 0$.
\end{theorem}

We refer to \cite[Definition 3.8]{patakfalvi2019ordinary} for a definition of $W\mathcal{O}$-rational singularities. It suffices to know that regular varieties have $W\mathcal{O}$-rational singularities, from which we obtain the following.

\begin{corollary}\label{split}
	Let $X$ be a uniruled, $F$-split surface over a perfect field of positive characteristic. If $K_{X} \equiv 0$ then $X$ has worse than canonical singularities.
\end{corollary}
\begin{proof}
	
	Suppose for contradiction that $X$ has canonical singularities. Then we can replace $X$ with its minimal resolution and suppose that $X$ is smooth. In particular it is Cohen-Macaulay and has $W\mathcal{O}$-rational singularities. We then apply \autoref{WO-unir} to obtain the result.
	
\end{proof}

\begin{lemma}\label{vol}\cite[Lemma 2.5]{jiang2018birational}
	Suppose $X$ is projective and normal, $D$ is an $\mathbb{R}$-Cartier divisor and $S$ is a basepoint free normal and prime divisor. Then for any $q >0$,
	\[\Vol(X,D+qS) \leq \Vol(X,D) + q\dim(X)\Vol(S,D|_{S}+qS|_{S}).\]
\end{lemma}	

We now collect the necessary results from the positive characteristic MMP and consider a few applications.

	\begin{theorem}\cite[Theorem 1.7]{birkar2017existence}, \cite[Theorem 1.2]{birkar2013existence} \label{Cone Theorem}
	Let $k$ be an algebraically closed field of characteristic $p>5$. 
	Let $(X, \Delta)$ be a three-dimensional klt pair over $k$, together with a projective morphism $X \to Z$ a quasi-projective $k$ scheme, then there exists a $(K_X+\Delta)$-MMP over $Z$ that terminates. 
	
	In particular, if $X$ is $\mathbb{Q}$-factorial, then 
	there is a sequence of birational maps of three-dimensional normal and $\mathbb{Q}$-factorial varieties:  
	\[
	X=:X_0 \overset{\varphi_0}{\dashrightarrow} X_1 \overset{\varphi_1}{\dashrightarrow} \cdots \overset{\varphi_{\ell-1}}{\dashrightarrow} X_{\ell}
	\]
	such that if $\Delta_i$ denotes the strict transform of $\Delta$ on $X_i$, then
	the following properties hold:  
	\begin{enumerate}
		\item 
		For any $i \in \{0, \ldots, \ell\}$, 
		$(X_i, \Delta_i)$ is klt and projective over $Z$.
		\item 
		For any $i \in \{0, \ldots, \ell-1\}$, 
		$\varphi_i\colon X_i \dashrightarrow X_{i+1}$ is either a $(K_{X_i}+\Delta_i)$-divisorial contraction over $Z$ or a $(K_{X_i}+\Delta_i)$-flip over $Z$. 
		\item 
		If $K_X+\Delta$ is pseudo-effective over $Z$, then $K_{X_{\ell}}+\Delta_{\ell}$ is nef over $Z$. 
		\item 
		If $K_X+\Delta$ is not pseudo-effective over $Z$, then 
		there exists a $(K_{X_{\ell}}+\Delta_{\ell})$-Mori fibre space $X_{\ell} \to Y$ over $Z$. 
	\end{enumerate}
\end{theorem}
	
	\begin{theorem}\cite[Theorem 10.4]{fujino2009fundamental}\label{dlt}
		Let $X$ be a normal quasi-projective variety of any dimension and characteristic for which the log MMP holds. Let $B$ be an effective divisor with $K_{X}+B$ $\mathbb{R}$-Cartier then there is a birational morphism $f\colon Y \to X$, called a dlt modification, such that the following holds:
		\begin{itemize}
			\item $Y$ is $\mathbb{Q}$-factorial;
			\item $a(E,X,B) \leq -1$ for every $f$ exceptional divisor $E$;
			\item If $B_{Y}=f^{-1}_{*}B' + \sum_{E \text{ exceptional}} E$ then $(Y,B_{Y})$ is dlt; and
			\item $K_{Y}+B_{Y}+F=f^{*}(K_{X}+B)$ where $F= \sum_{E\colon a(E,X,B)<-1} -(a(E,X,B)+1)E$.
		\end{itemize}
	where $B'$ has coefficient $\min\{\text{Coeff}_{E}(B),1\}$ at each $E$. Further if $(X,B)$ is a log pair then $F$ is exceptional.
	\end{theorem}

	\begin{theorem}[Nlc Cone Theorem]\label{NLCT}
	Let $(X,\Delta)$ be a threefold $\mathbb{Q}$-pair, over a closed field of characteristic $p > 5$. Then write $\overline{NE}(X/T)_{nlc}$ for the cone spanned by curves contained in the non log canonical locus of $X$. Then we have the following decomposition
	\[\overline{NE(X)}(X)=\overline{NE(X)}_{K_{X}+B \geq 0}+ \overline{NE(X)}_{nlc}+R_{i}\]
	where $R_{i}$ are extremal rays with $R_{i} \cap\overline{NE(X)}_{nlc}=\{0\}$, generated by curves $C_{i}$ such that $0> (K_{X}+B).C_{i} \geq -6$.
\end{theorem}
\begin{proof}
	If $(X,\Delta)$ is dlt this is part of the usual Cone Theorem \cite[Theorem 1.1]{birkar2017existence}.

	Suppose next that $\Delta=B+F$ where $(X,B)$ is dlt and $F$ has support contained in $\lfloor B \rfloor$. Note that if $C$ is an irreducible curve with $F.C <0$ then $C \subseteq F$. Therefore any effective curve $C$ can be written $C=C_{0} +C_{F}$ where $F.C_{0}\geq 0$ and $C_{F} \subseteq F$. Thus by compactness of the unit ball in a finite dimensional vector space, any $[\gamma] \in \overline{NE}(X/T)$ can be written $[\gamma] = [\gamma_{0}] + [\gamma_{F}]$ with $F.\gamma_{0} \geq 0$ and $[\gamma_{F}] \in \overline{NE}(F/T)$ in the same fashion.
	
	Take any $K_{X}+\Delta$ negative extremal ray $L$. Take a non-zero $[\gamma] \in L$, then as $L$ is extremal we have $[\gamma_{F}],[\gamma_{0}] \in L$. If $[\gamma_{F}] \neq 0$ then $L \subseteq \overline{NE}(F/T)$. Otherwise if $[\gamma_{F}]=0$ then $L$ is $K_{X}+B$ negative. Hence we can conclude the result from the Cone Theorem for dlt pairs.

	Suppose finally that $X$ is not dlt. Let $\pi \colon Y \to X$ be a dlt modification of $(X,B)$ with $(Y,B_{Y})$ dlt and $K_{Y}+B_{Y}+F=\pi^{*}(K_{X}+B)$. Take any $K_{X}+B$ negative extremal ray, $L$, such that $L \cap\overline{NE(X)}_{nlc}=\{0\}$. Take any class $\gamma$ with $[\gamma] \in L\setminus \{0\}$ and choose $[\gamma'] \in \overline{NE}(Y/T)$ with $f_{*}[\gamma']=[\gamma]$. Then by the projection formula we have that $(K_{Y}+B_{Y}+F).\gamma'=(K_{X}+B).f_{*}\gamma'=(K_{X}+B).\gamma < 0$. 
	
	From above, we can write $\gamma'=C_{0}+C_{F}+ \sum \lambda_{i}C_{i}$ where $\lambda_{i} >0$, $(K_{Y}+B_{Y}+F).C_{0} \geq 0$ $C_{F} \in \overline{NE}(F/T)$ and the $C_{i}$ each generate $(K_{Y}+B_{Y}+F)$ negative extremal rays with $-(K_{Y}+B_{Y}+F).C_{i} \leq 6$. 	
	From our choice of $R$ we must have $f_{*}C_{0}=f_{*}C_{F}=0$ and hence it follows that $[f_{*}C_{k}] \in R\setminus \{0\}$ for some $k$. Thus $(K_{X}+B).f_{*}C_{k}=(K_{Y}+B_{Y}+F).C_{k} \geq -6$.
	
	Since each $R$ is the pushforward of a $(K_{Y}+B_{Y})$ negative extremal ray, there are only countably many generating curves $C_{i}$ and they cannot accumulate in $(K_{X}+\Delta)_{< 0}$ else they would accumulate on $Y$ also.
\end{proof}

	\begin{lemma}
	Let $X$ be a normal curve over any field and $\Delta \geq 0 $ be a divisor with $-(K_{X}+\Delta)$ big and nef. Then the non-klt locus of $\Delta$ is either empty or geometrically connected. 
\end{lemma}

\begin{proof}
	If $-(K_{X}+\Delta)$ is big and nef then so is $-K_{X}$. After base changing to $H^{0}(X,\ox)$ if necessary we have $\deg K_{X} = -2$ by \cite[Corollary 2.8]{tanaka2018minimal} giving that $ \deg \Delta <2$. The non-klt locus of $(X,\Delta)$ is precisely the support of $\lfloor \D \rfloor$ and hence can contain at most one point.
\end{proof}

\begin{theorem}\cite[Theorem 5.2]{tanaka2018minimal}\label{Tcl}
	Let $(X,\Delta)$ be a surface log pair over any field $\kappa$. Let $\pi\colon X \to S$ be a morphism of $\kappa$ schemes with $\pi_{*}\ox =\mathcal{O}_{S}$. Suppose that $-(K_{X}+\Delta)$ is $\pi$-nef and $\pi$-big, then for any $s \in S$ $X_{S}\cap \nklt(X,\Delta)$ is either empty or geometrically connected. 
\end{theorem}

\begin{theorem}[Weak Connectedness Lemma]\label{WCL}
	Let $X$ be a threefold over any closed field $\kappa$ of characteristic $p>5$ together with $\Delta\geq 0$ on $X$ such that $K_{X}+\Delta$ is $\mathbb{R}$-Cartier. Suppose that $-(K_{X}+\Delta)$ is ample, then $\nklt(X,\Delta)$ is either empty or connected. 
\end{theorem}
\begin{proof}
	If $(X,\D)$ is klt the result is trivially true, so suppose otherwise.
	
	Let $(Y,\D_{Y}) \to (X,\D)$ be a dlt modification. Then $-L:=K_{Y}+\D_{Y}+F=f^{*}(K_{X}+\D)$ with $(Y,\D_{Y})$ dlt and $L$ nef and big. We may further write $L=A+E$ with $A$ ample and $E$ effective and exceptional over $X$. In particular $E$ has support contained inside $S_{Y}=\lfloor \D_{Y} \rfloor$. Note that $S_{Y}$ maps surjectively onto $\nklt(X,\D)$ so it is sufficient to show that $S_{Y}$ is connected.
	
	Take a general $G_{Y} \sim \epsilon A +(1-\epsilon) L-\delta S_{Y}$, then for small $\delta$ we may assume $G_{Y}$ is ample, and hence further that $(X,\D_{Y}+G_{Y})$ is dlt. Write $K_{Y}+\D_{Y}+G_{Y}\sim - P_{Y}=-(\epsilon E + F + \delta S_{Y})$ and note $\text{Supp}(P_{Y})=S_{Y}$. In particular $K_{Y}+\D_{Y}+G_{Y}$ is not pseudo-effective and hence we may run a $(Y,\D_{Y}+G_{Y})$ LMMP which terminates in a Mori fibre spaces $Y' \to Z$. By the arguments of \cite[Theorem 9.3]{birkar2013existence} on the induced pair $(Y',\D_{Y'})$, $\nklt(Y',\D_{Y'})=\text{Supp}(\lfloor \D_{Y'} \rfloor)=\text{Supp}(P_{Y'})$ has the same number of connected components as $\nklt(X,\Delta)$, so it suffices to prove the result here.
	
	Suppose first that $\dim Z=0$. Then $\rho(Y')=1$. In particular if $D,D'$ are effective and $H$ ample, then $H.D.D' >0$, so certainly $D.D'>0$. Thus $P_{Y'}$ cannot have disconnected support.
	
	Suppose next that $\dim Z > 0 $. Let $T$ be the generic fibre. We must have $P_{Y'}|_{T}> 0$ since $Y' \to Z$ is a $P_{Y'}\sim -(K_{Y'}+\Delta_{Y'}+G_{Y'})$ positive contraction. However $P_{Y'}$ has the same support as $\lfloor \D_{Y'} \rfloor$ so at least one connected component must dominate $Z$. Suppose then, for contradiction, there is a second connected component. Clearly it must also dominate $Z$, else it could not possibly be disjoint from the first. Consider then $(T,\D_{T}=\D_{Y'}|_{T})$. Since $T \to Y'$ is flat, the pullback of $\D_{Y'}$ is just the inverse image, and in particular $\lfloor \D_{T} \rfloor$ contains the pullback of both connected components. Suppose $R$ is the extremal ray whose contraction induces the Mori fibration. Then we have $-(K_{Y'}+\D_{Y'}+G_{Y'}).R >0$, but since $R$ is spanned by a nef curve, as contracting it defines a fibration, and $G_{Y'}$ is effective, we must have $G_{Y'}.R \geq 0$. Hence in fact $-(K_{Y'}+\D_{Y'}).R >0$ also, and so $-K_{T}+\D_{T}$ is ample. Then, however, the non-klt locus of $(T,\D_{T})$ must be connected, a contradiction.
\end{proof}

\begin{lemma}\label{cc} \cite[Proposition 4.37]{kollar2013singularities}
		Suppose that $(S,B)$ is a klt surface and $(K_{S}+B+D) \sim 0$ for $D$ effective, integral and disconnected, then $D$ has exactly two connected components.
\end{lemma}

Finally we collect some needed Bertini type theorems. 
	
\begin{theorem}\cite[Theorem 1]{tanaka2017semiample}
	Let $(X,\Delta)$ be a log canonical (resp. klt) pair over an algebraically closed field where $\Delta$ is an effective $\mathbb{Q}$-divisor. Suppose $D$ is a semiample divisor on $X$ then there is an effective divisor $D'\sim D$ with $(X,\Delta+D')$ log canonical (resp. klt).
\end{theorem}

\begin{corollary}\label{average}
	Suppose that $(X,\Delta)$ is a sub klt pair over an algebraically closed field together with $D$ a divisor on $X$ and $\pi\colon (X',\Delta') \to X$ a log resolution of $(X,\Delta)$. Further assume that there is some $D'$ on $X'$ with $\pi_{*}D'=D$, $-(K_{X'}+\Delta'+D')$ $\pi$-nef, $(X,\Delta')$ sub klt and $D'$ semiample. Then there is $E \sim D$ on $X$ effective with $(X,\Delta+E)$ sub klt. If in fact $(X,\Delta)$ is $\epsilon$-klt then we may choose $E$ such that $(X,\Delta+E)$ is also.
\end{corollary}
\begin{proof}
	We may write $\Delta'=\Delta_{p}-\Delta_{n}$ as the difference of two effective divisors. Since $(X',\Delta')$ is log smooth we must have that $(X',\Delta_{p})$ is klt. Thus by the proceeding theorem we have that there is some $E' \sim D'$ with $(X',\Delta_{p}+E')$ klt. Then we must also have that $(X',\Delta'+E')$ is sub klt 
	Write $E=\pi_{*}E'$, then $R=\pi^{*}(K_{X}+\Delta+E)- (K_{X'}+\Delta'+E')\equiv_{f}-(K_{X'}+\Delta'+D')$ is $\pi$-nef and exceptional. Hence by the negativity lemma we have that $-R$ is effective, and $\pi^{*}(K_{X}+\Delta+E) \leq (K_{X'}+\Delta'+E')$ giving that $(X,\Delta +E)$ is klt.
	
	If $(X',\Delta)$ is $\epsilon$-klt then so is $(X',\Delta_{p})$. Let $\delta =\min (1-\epsilon-c_{i})$ where $c_{i}$ are the coefficients of $\Delta_{p}$ and take $m \in \mathbb{N}$ such that $\frac{1}{m} < \delta$. Applying the previous theorem to $mD'$ instead of $D'$, yields $E'' \sim mD$ with $(X',\Delta'+E'')$ klt. Taking $E'=\frac{1}{m}E$ then continuing as above gives the required divisor. 
\end{proof}

\begin{theorem}\cite[Corollary 1.6]{patakfalvi2017singularities}\label{smoothness}
	Let $f\colon X \to Z$ be a projective fibration of relative dimension $2$ from a terminal variety with $f_{*}\ox=\mathcal{O}_{Z}$ over a perfect field of positive characteristic $p > 7$, such that $-K_{X}$ is ample over $Z$. Then a general fibre of $f$ is smooth.
\end{theorem}

\begin{theorem}[Bertini for residually separated morphisms]\cite[Theorem 1]{cumino1986axiomatic}\label{Bertini}
	Let $f\colon X \to \mathbb{P}^{n}$ a residually separated morphism of finite type from a smooth scheme over a closed field. Then the pullback of a general hyperplane $H$ on $\mathbb{P}^{n}$ is smooth.
\end{theorem}

Here residually separated means that the induced map on residue fields $\ox[X,x] \to \ox[\mathbb{P}^{n}, f(x)]$ is a separable extension.

\section{Conic Bundles}\label{S-CB}

	In this section the ground field will always be algebraically closed of characteristic $p> 0$. In some results we put additional restrictions on the characteristic, most often that $p \neq 2$.
	We start with some useful results on finite morphisms and klt singularities.
	
	\begin{definition}
		Take a finite, separable and dominant morphism of normal varieties $f\colon X \to Y$.
		
		If $D$ is a divisor on $Y$ then $f$ is said to be \emph{tamely ramified over $D$} if for every prime divisor $D'$ lying over $D$ the ramification index is not divisible by $p$ and the induced residue field extension is separable.
		
		Moreover $f$ is said to be \emph{divisorially tamely ramified} if for any proper birational morphism of normal varieties $Y' \to Y$ we have the following. If $X' \to X$ is the normalisation of the base change $X\times_{Y}Y'$, and $f'\colon X'\to Y'$  the induced map, then $f'$ is tamely ramified over every prime divisor in $Y'$.
		
		If instead $f$ is generically finite, we say it is divisorially tamely ramified if the finite part of its Stein factorisation is so. Equally if either of $X$ or $Y$ is not normal, $f\colon X \to Y$ is said to be divisorially tamely ramified if the induced morphism on their normalisations is.
	\end{definition}
	
	If $f$ is generically finite of degree $d <p$ then it is always divisorially tamely ramified. If $D'$ lies over $D$ then both the ramification index, $r_{D'}$ and the inertial degree, $e_{D'}$ are bounded by $d$, in fact $d= \sum_{f(D')=D} r_{D'}e_{D'}$ by multiplicativity of the norm. This remains the case on any higher birational model.
	
	\begin{lemma}
		Let $f\colon Y \to X$ be a dominant, separable, finite morphism of normal varieties. Suppose that $K_{X}$ is $\mathbb{Q}$-Cartier then $K_{Y}=f^{*}K_{X}+\Delta$ where $\Delta \geq 0$. Further if $f$ is divisorially tamely ramified, then for $Q \in Y$ a codimension $1$ point lying over $P \in X$ we have $\text{Coeff}_{Q}(\Delta)=r_{Q}-1$ where $r_{Q}$ is the degree of $f|_{Q}\colon Q \to P$.
	\end{lemma}
	\begin{proof}
		By localising at the codimension $1$ points of $X$ we reduce to the case of Riemann-Hurwitz-Hasse to see that $\Delta$ exists as required and $\text{Coeff}_{Q}(\Delta)=\delta_{Q}$ where $\delta_{Q} \geq r_{Q}-1$ with equality when $p\nmid r_{q}$. In particular when $f$ is divisorially tamely ramified, we ensure $\delta_{Q}=r_{Q}-1$.
	\end{proof}
	
	The singularities of the domain and image of a finite divisorially tame ramified morphism are closely connected, as the following lemma shows.
	
	\begin{lemma}\cite[Proposition 3.16]{kollar1997singularities} \label{finite adjunction}
		Let $f\colon X' \to X$ be a dominant, divisorially tamely ramified, finite morphism of normal varieties of degree $d$. Fix $\Delta$ on $X$ with $K_{X}+\Delta$ $\mathbb{Q}$-Cartier. Write $K_{X'}+\Delta'=f^{*}(K_{X}+\Delta)$ then the following hold:
		\begin{enumerate}
			\item $1+\text{TDisc}(X,\Delta) \leq 1+\text{TDisc}(X',\Delta') \leq d(1+\text{TDisc}(X,\Delta))$.
			\item $(X,\Delta)$ is sub klt (resp. sub lc) iff $(Y,\Delta')$ is sub klt (resp. sub lc).
		\end{enumerate}
	\end{lemma}

\begin{proof}
	
	By restricting to the smooth locus of $X$, which contains all the codimension $1$ points of $X$, we may suppose that $K_{X}$ is Cartier and apply the previous lemma. Hence we get $\Delta'=f^{*}(K_{X}+\Delta)-K_{X'}$ where for $Q\in X'$ lying over $P\in X$ we have $\text{Coeff}_{Q}(\Delta')=r_{Q}(\text{Coeff}_{P}(\Delta))-(r_{Q}-1)$.
	
	Suppose that we have proper birational morphisms $\pi\colon Y \to X$ and we write $Y'$ for the normalisation of $Y\times_{X} X'$ so that we have the following diagram.
	
	\[\begin{tikzcd}
	Y' \arrow[d, "\pi'"] \arrow[r, "g"] & Y \arrow[d, "\pi"] \\
	X' \arrow[r, "f"]                   & X                 
	\end{tikzcd}\]
	Let $E'$ be a divisor on $Y'$ exceptional over $X'$ and $E$ the corresponding divisor on $Y$.
	
	At $E'$ we can write $$K_{Y'}= \pi'^{*}(K_{X'}+\Delta')+a(E',X',\Delta')E'=g^{*}\pi^{*}(K_{X}+\Delta)+a(E',X',\Delta')E'$$
	essentially by definition. Conversely however we have $K_{Y'}=g^{*}K_{Y}+\delta_{E'}E'$ which may be rewritten as 
	$$K_{Y'}=g^{*}(\pi^{*}(K_{X}+\Delta)+a(E,X,\Delta)E)+\delta_{E'}E'.$$
	
	In particular equating the two descriptions, as $\delta_{E'}=r_{E'}-1$ by \autoref{finite adjunction}, we have that
	\[r_{E'}a(E,X,\Delta)+(r_{E'}-1)=a(E',X',\Delta')\]
	and thus $a(E,X,\Delta)+1=\frac{1}{r_{E'}}(a(E',X',\Delta')+1)$ with $1 \leq r_{E'} \leq d$.
	
	Since, by a theorem of Zariski \cite[Theorem VI.1.3]{kollar1999rational}, every valuation with center on $X'$ is realised by some birational $Y' \to X'$ occurring as a pullback of a birational morphism $Y \to X$, this is sufficient to show that $1+\text{TDisc}(X,\Delta) \leq 1+\text{TDisc}(X',\Delta') \leq d(1+\text{TDisc}(X,\Delta))$. The second part then follows.
\end{proof}

We will be interested in conic bundles satisfying certain tameness criteria. This in turn will allow us to control the singularities arising on the base of the fibration. This is done in \autoref{cbf}. 

\begin{definition}
	A \emph{conic bundle} is a threefold sub pair $(X,\Delta)$ equipped with a morphism $f\colon X \to Z$ where Z is a normal surface, $f_{*}\ox=\mathcal{O}_{Z}$, the generic fibre is a smooth rational curve and $(K_{X}+\Delta)=f^{*}D$ for some $\mathbb{Q}$-Cartier divisor on $X$. We will call it \emph{regular} if $X$ and $Z$ are smooth and $f$ is flat; and \emph{terminal} if $X$ is terminal and $f$ has relative Picard rank $1$. Further we call it (sub) $\epsilon$-klt or log canonical if $(X,\Delta)$ is.
	
	If each horizontal component of $\Delta$ is effective and divisorially tamely ramified over $Z$ then the conic bundle is said to be \emph{tame}.
	
	For $P$ a codimension $1$ point of $Z$ we define $$d_{P}=\max\{t\colon  (X,\Delta+tf^{*}P) \text{ is lc over the generic point of } P\}.$$
	The \emph{discriminant divisor} of $f\colon X \to Z$ is $D_{Z}=\sum_{P \in X}(1-d_{P})P$.
	The \emph{moduli part} $M_{Z}$ is then given by $D-D_{Z}-K_{Z}$.
\end{definition}

In positive characteristic the discriminant divisor is not always well defined for a general fibration, it may be that $d_{P} \neq 1$ for infinitely many $P$. This can be caused by either a failure of generic smoothness or inseparability of the horizontal components of $\Delta$ over the base.

Suppose, however, that $(X,\Delta) \to Z$ is a tame conic bundle. We may take a log resolution $X' \to X$ as this does not change $d_{P}$ and is still a tame conic bundle by \autoref{tame base change}. Thus we may suppose that $\Delta$ is an SNC divisor and hence near $P$, $\Delta+f^{*}P$ is also SNC for all but finitely many $P$, by generic smoothness of the fibres and as the horizontal components are divisorially tamely ramified over $Z$. Hence in fact $D_{Z}$ is well defined in this case.

\begin{lemma}\label{tame base change}
	Let $f\colon (X,\Delta) \to Z$ be a tame conic bundle, and $X' \to X$ either a birational morphism from a normal variety or the base change by a divisorially tamely ramified morphism from a normal variety $g\colon Z' \to Z$. Then there is $\Delta'$ with $(X',\Delta')$ a tame conic bundle over $Z$ or $Z'$ as appropriate. Moreover in this case $X' \to X$ is also divisorially tamely ramified.
\end{lemma}

\begin{proof}
	If $\pi\colon X' \to X$ is a birational morphism with $K_{X'}+\Delta'=\pi^{*}(K_{X}+\Delta)$ then the only horizontal components of $\Delta'$ are the strict transforms of horizontal components of $\Delta$. Take such a component $D'$ then, normalising if necessary, it factors $D' \to D \to Z$ with $D \to Z$ divisorially tamely ramified but then it must itself be divisorially tamely ramified.
	
	Suppose then $g\colon Z' \to Z$ is generically finite. From above, and by Stein factorisation we may freely suppose that $g$ is finite. Then the base change morphism $g'\colon X' \to X$ is a finite morphism of normal varieties and we may induce $\Delta'$ with $g'^{*}(K_{X}+\Delta)=K_{X'}+\Delta'$. Again the horizontal components of $\Delta'$ are precisely the base changes of the horizontal components of $\Delta$. 
	
	It suffices to show then that if $D$ is a horizontal divisor on $X$ such that $D \to Z$ is divisorially tamely ramified then $D' \to Z'$, the base change, is also divisorially tamely ramified. Certainly $D' \to Z'$ is still separable. Suppose $C$ is any curve on $Z$ and $C'$ a curve on $Z'$ lying over it. In turn take any $C_{D'}$ lying over $C'$ on $D'$. Then $C_{D'}$ is the base change of some $C_{D}$. Since $C_{D} \to C$ is separable, so too is $C_{D'} \to C'$. Equally as the ramification indices of $C', C_{D}$ are not divisible by $p$, neither can the ramification index of $C_{D'}$ over $C_{D}$ be. This same argument holds after base change by any higher birational model of $Z$, and by \cite[Theorem VI.1.3]{kollar1999rational} every valuation with centre on $Z'$ is can be realised on the pullback of some such model. Thus $D' \to Z'$ is divisorially tamely ramified and hence $(X',\Delta') \to Z'$ is tame.
	
	It is enough to show that $X' \to X$ is divisorially tamely ramified after base changing by a higher birational model of $Z$. In particular, after taking a flatification we may assume $f\colon X \to Z$ is flat. Now suppose $D$ is a divisor on $X$, lying over some curve $C$ on $Z$. We have $f^{*}C=\sum E_{i}$ with $E_{0}=D$. Let $C_{j}$ be the curves lying over $C$ in $Z'$, then if $E_{i,j}$ are the divisors lying over $E_{i}$, for some fixed $i$, they are in one-to-one correspondence with the $C_{j}$. We have $g'^{*}f^{*}C=\sum r_{i,j}E_{i,j}=\sum_{j} r_{i}\sum _{i}E_{j}$ and thus none of the $r_{i,j}$, in particular the $r_{0,j}$ are divisible by $p$. Moreover the $E_{0,j} \to E_{0}$ must be separable since the $C_{j} \to C$ are.
	
	The same holds after taking a higher birational model of $X$, and thus $X' \to X$ is divisorially tamely ramified as claimed.
	
	\end{proof}

In practice we deal exclusively with tame conic bundles arising in the following fashion.

\begin{lemma}\label{S2}
	Suppose that $(X,\Delta)$ is klt and LCY, equipped with a Mori fibre space structure over a surface $Z$ and the horizontal components of $\Delta$ have coefficients bounded below by $\delta$. Then if $X$ is defined over a field of characteristic $p > \frac{2}{\delta}$, $f\colon (X,\Delta) \to Z$ is a tame conic bundle.
\end{lemma}
\begin{proof}
	
	Since $\delta <1$, the characteristic is larger than $2$ and the general fibre is necessarily a smooth rational curve, in particular $X$ is a conic bundle. Let $G$ be the generic fibre, so that $(G,\Delta_{G})$ is klt and $G$ is also smooth rational curve. Then if $D$ is some horizontal component of $\Delta$ the degree of $f\colon D \to Z$ is precisely the degree of $D|_{G}$. However $\deg \delta D|_{G} < \deg \Delta|_{G} =-2$ and thus $\deg D < p$. Replacing $D$ by its normalisation, $D'$ does not change the degree, so $D'\to Z$ has degree $<p$ and thus is divisorially tamely ramified.
\end{proof}

\begin{remark}
	One might be tempted to ask if this bound could be further improved for $\epsilon$-klt pairs, $(X,\Delta)$. In this case we have $(G,\Delta_{G})$ is $\epsilon$-klt and so one might attempt to use a bound of the form $p > \frac{1-\epsilon}{\delta}$ to prevent any component of $\Delta$ mapping inseparably onto the base. It does not seem however that such a bound would ensure that every component is divisorially tamely ramified and there may be wild ramification away from the general fibre. 
\end{remark}

\begin{theorem}\label{cbf}
	Let $f\colon (X,\Delta) \to Z$ be a sub $\epsilon$-klt, tame conic bundle. Then for some choice of $M\sim_{\mathbb{Q}} M_{Z}$ we have $(Z,D_{Z}+M)$ sub $\epsilon$-klt. If in fact $\Delta \geq 0$, we may take $D_{Z},M$ to be effective also.
\end{theorem}
\begin{remark}
	The implicit condition that $(X,\Delta)$ is a threefold pair is necessary only in that it assures the existence of log resolutions. This result holds in dimension $d$ so long as the existence of log resolutions of singularities holds in dimensions $d,d-1$.
\end{remark}

We will prove this in several steps. First we consider the case that $\Delta^{h}$, the horizontal part of $\Delta$, is a union of sections of $f$. In this setting we have an even stronger result. After moving to a higher birational model, we have that $(Z,D_{Z})$ is klt and $M_{Z}$ is semiample.

\begin{lemma}\label{ShokurovAdjunction}
	Suppose that $f\colon (X,\Delta) \to Z$ is a sub $\epsilon$-klt conic bundle with $\Delta^{h}$ effective and with support that is generically a union of sections of $f$, then there is $\pi \colon Z'\to Z$ a birational morphism with $(Z',D_{Z'})$ sub $\epsilon$-klt and $M_{Z'}$ semiample. In particular for some choice of $M\sim M_{Z'}$ we have $(Z,D_{Z}+\pi_{*}M)$ sub $\epsilon$-klt. 
\end{lemma}
\begin{proof}
	This result is well known and essentially comes from \cite{prokhorov2009towards}. Details specific to positive characteristic can be found in \cite[Section 4]{das2016adjunction}, \cite[Lemma 3.1]{witaszek2017canonical} and \cite[Lemma 6.7]{cascini2013base}

	We sketch, some key points of the proof.
	
	Since generically $X \to Z$ is a $\mathbb{P}^{1}$ bundle and the horizontal part of $\Delta$ is a union of sections, we induce a rational map $\phi\colon Z \dashrightarrow \overline{\mathcal{M}}_{0, n}$, the moduli space of $n$-pointed stable curves of genus $0$. By taking appropriate resolutions we may suppose that $(X,\Delta)$ is log smooth, $Z$ is smooth and $\phi$ is defined everywhere on $Z$. Blowing down certain divisors on the universal family over $\overline{\mathcal{M}}_{0, n}$ and pulling back to $Z$ we may further assume that $X\to Z$ factors through a $\mathbb{P}^{1}$ bundle over $Z$ via a birational morphism.
	
	Then working locally over each point of codimension $1$ and applying $2$ dimensional inversion of adjunction, we see that in fact $D_{Z}$ is determined by the vertical part of $\Delta$, indeed $\Delta^{V}=f^{*}D_{Z}$, and that $M_{Z}$ is the pullback of an ample divisor on $\overline{\mathcal{M}}_{0, n}$ by $\phi$. In particular $M_{Z}$ is semiample and $D_{Z}$ takes coefficients in the same set as $\Delta^{V}$, therefore they are bounded above by $1-\epsilon$.
	
	From the following lemma, we see that in fact we may further suppose that $(Z,D_{Z})$ is log smooth. Since if $\pi\colon (Z',\Delta')\to Z$ is a log resolution of $(Z,D_{Z})$ we have $K_{Z'}+\Delta'=\pi^{*}(K_{Z}+D_{Z})$, $\pi^{*}M_{Z}=M_{Z'}$ and $K_{Z'}+D_{Z'}+M_{Z'}=\pi^{*}(K_{Z}+D_{Z}+M_{Z})=K_{Z'}+\Delta'+M_{Z'}$, giving $D_{Z'}=\Delta'$ as required. In particular then \autoref{average} gives that $(Z,D_{Z}+M_{Z})$ is sub $\epsilon$-klt.
\end{proof}

\begin{lemma}
	Suppose that $Z$ is as given above and $Z'\to Z$ is the birational model found in the proof with $M_{Z'}$ semiample. Suppose further that $Y$ is a normal variety admitting a birational morphism $\pi\colon Y \to Z'$. If $M_{Y}$ is the moduli part coming from the induced conic bundle $X_{Y} \to Y$ then $\pi^{*}M_{Z'}=M_{Y}$. 
\end{lemma}
\begin{proof}
	
	Let $\phi\colon  Z' \to \overline{\mathcal{M}}_{0, n}$ and $\chi\colon  Y \dashrightarrow \overline{\mathcal{M}}_{0, n}$ be the rational maps induced by the base changes of $X\to Z$. By assumption $\phi$ is a morphism.
	
	Although $\chi$ is a priori defined only on some open set, it must factor through $\phi$ whenever it is defined, and hence extends to a full morphism $\chi=\phi \circ \pi$.
	
	Write then that $M_{Z'}=\phi^{*}A$ and $M_{Y}=\chi^{*}A'$. A more careful study of the proof of the previous result would give $A=A'$ and the result follows. However for simplicity one can also note that $M_{Z'}=\pi_{*}M_{Y}=\pi_{*}\chi^{*}A'=\phi^{*}A'$, so that $M_{Y}=\pi^{*}\phi^{*}A'=\pi^{*}M_{Z'}$.		
\end{proof}

	We now reduce from the general case of \autoref{cbf} to the special case of \autoref{ShokurovAdjunction} to prove the theorem. This requires the following lemma, due essentially to Ambro.

\begin{lemma}\cite[Theorem 3.2]{ambro1999adjunction}
	Suppose that $f\colon (X,\Delta) \to Z$ is a tame conic bundle. Let $g\colon Z' \to Z$ be a finite, divisorially tamely ramified morphism of normal varieties and $(X',\Delta') \to Z'$ the induced fibration. Then $(X',\Delta') \to Z$ is tame and $g^{*}(K_{Z}+D_{Z})=K_{Z'}+D_{Z'}$ for $D_{Z'}$ the induced discriminant divisor of $(X',\Delta') \to Z'$.
\end{lemma}

\begin{proof}
	By \autoref{tame base change}, $(X',\Delta') \to Z'$ is tame and hence $D_{Z'}$ is well defined by the discussion proceeding \autoref{tame base change}.
	
	It remains to show that $g^{*}(K_{Z}+D_{Z})=K_{Z'}+D_{Z'}$. To see this fix $Q$ a prime of $Z'$ and write $r_{Q}$ for the degree of the induced map onto some $P$ a prime of $Z$. 
	
	From the proof of \autoref{finite adjunction} we see that if $K_{Z'}+B=g^{*}(K_{Z}+D_{Z})$ then $1-\text{Coeff}_{Q}(B)=r_{Q}(\text{Coeff}_{P}(D_{Z})-1)$. In particular then it suffices to show that $d_{Q}=r_{Q}d_{P}$. We consider two cases.
	
	Suppose that $c \leq d_{P}$. Then we have $(X,\Delta+cf^{*}P)$ log canonical over $P$. Hence $(X',\Delta'+g'^{*}f^{*}P=\Delta+cf'^{*}g^{*}P)$ is also log canonical by the \autoref{finite adjunction}. But $f'^{*}g^{*}P \geq f'^{*}r_{Q}Q$ so it must be that $d_{Q} \geq r_{Q}c$. Hence in fact $d_{Q} \geq r_{Q}d_{P}$.
	
	Conversely if $c \geq d_{P}$ then,$(X,\Delta+cf^{*}P)$ is not log canonical over $P$. In particular replacing $X$ with a suitable birational model $X'' \to X$ we suppose that there is some prime $E$ of $X$ with $f_{E}=P$ and $\text{Coeff}_{E}(\Delta+cf^{*}P) < -1$. Similarly there is $E'$ on $X'$ with $g'(E')=E$ and $f'(E')=Q$ which also has $\text{Coeff}_{E}(\Delta'+cg'^{*}f^{*}P) < -1$ but $\text{Coeff}_{E}(cg'^{*}f^{*}P)=\text{Coeff}_{E}(cf^{*}r_{Q}P)$ and hence $c \geq rd_{Q}$. Thus we have the equality $d_{Q}=r_{Q}d_{Q}$.
\end{proof}

Note that in the setup above $g^{*}(K_{Z}+D_{Z}+M_{Z})=K_{Z'}+D_{Z'}+M_{Z'}$ so we must have that $M_{Z'}=g^{*}M_{Z}$.

\begin{lemma}
	Suppose that $f\colon X \to Z$ is a tame conic bundle. Then there is a finite, divisorially tamely ramified morphism $g\colon Z' \to Z$ with $g^{*}(K_{Z}+D_{Z}+M_{Z})=K_{Z'}+D_{Z'}+M_{Z'}$ and a birational morphism $h\colon Z'' \to Z'$ such that $M_{Z''}$ is semiample.
\end{lemma}
\begin{proof}
	Let $D$ be any horizontal component of $\Delta$ which is not a section of $f$ then $f$ restricts to a divisorially tamely ramified morphism $D \to Z$. After replacing $D$ with its normalisation and Stein factorising, we may suppose that $D\to Z$ is finite with $D$ normal. Taking the fibre product of $X \to Z$ with the normalisation $\tilde{D}$ of $D$ we find $X' \to \tilde{D}$ satisfying the initial conditions but with the one component of $\Delta$ is now  generically a section.
	
	 In this fashion, we eventually get to $Z' \to Z$ with $g^{*}(K_{Z}+D_{Z}+M_{Z})=K_{Z'}+D_{Z'}+M_{Z'}$ and all the horizontal components of $\Delta$ being generically sections. Hence we may apply \autoref{ShokurovAdjunction} to give the result.
\end{proof}

\begin{proof}[Proof of \autoref{cbf}]
	Take $f\colon (X,\Delta) \to Z$ as given. Then we have $g\colon Z' \to Z$ and $h\colon Z''\to Z'$ as above. Write $d$ for the degree of $g$.
	Fix $B_{Z''}\sim M_{Z''}$ making $(Z'',D_{Z''}+B_{Z''})$ sub klt. Write $B_{Z}=\frac{1}{d}g_{*}h_{*}B_{Z''}$. It is sufficient to show that $(Z,D_{Z}+B_{Z})$ is sub $\epsilon$-klt since $B_{Z} \sim M_{Z}$ is always effective and $D_{Z} \geq 0$ whenever $\Delta$ is.
	
	Let $Y \to Z$ be a log resolution of $(Z,D_{Z}+B_{Z})$ and take $Y',Y''$ appropriate fibre products to form the following diagram.
	
	\[\begin{tikzcd}
	Y'' \arrow[r, "\pi''"] \arrow[d, "h'"] & Z'' \arrow[d, "h"] \\
	Y' \arrow[r, "\pi'"] \arrow[d, "g'"]   & Z' \arrow[d, "g"]  \\
	Y \arrow[r, "\pi"]                     & Z                 
	\end{tikzcd}\]

	We have that $M_{Y''}=\pi''^{*}M_{Z''}$, so write $B_{Y''}=\pi''^{*}B_{Z''}$ and $\frac{1}{d}g'_{*}h'_{*}B_{Y''}=B_{Y}$. Then we must have that $\pi_{*}B_{Y}=B_{Z}$ and $K_{Y}+D_{Y}+B_{Y}\sim \pi^{*}(K_{Z}+D_{Z}+B_{Z})$. Note further that $\pi^{*}B_{Z}$ and $B_{Y}$ differ only over the exceptional locus, hence $B_{Y}$ has SNC support. Indeed $D_{Y}+B_{Y}$ has SNC support. Further since $(Y'',D_{Y''}+B_{Y''})$ is sub $\epsilon$-klt and $g'_{*}h'_{*}(D_{Y''}+B_{Y''})=d(D_{Y}+B_{Y})$ it must be that $D_{Y}+B_{Y}$ have coefficients strictly less than $1-\epsilon$, thus $(Y,D_{Y}+B_{Y})$ is sub $\epsilon$-klt and therefore so is $(Z,D_{Z}+B_{Z})$.
\end{proof}

\subsection{Generic smoothness}

We will also need to consider the pullbacks of very ample divisors on the base of a suitably smooth conic bundle. This is done to obtain an adjunction result which is required in the next section. We work here under the assumption the ground field is closed of positive characteristic $p>2$. This requirement on the characteristic is due entirely to the following lemma.

	\begin{lemma}\label{disc-smooth}
	Let $(X,\Delta) \to Z$ be a regular conic bundle. Then there is some, possibly reducible, curve $C$ on $Z$ such that for any $P \in Z$ the fibre, $F_{P}$, over $P$ is determined as follows:
	\begin{enumerate}
		\item If $P \in Z\setminus C$ then $F_{p}$ is a smooth rational curve.
		\item If $P \in C\setminus \Sing(C)$ then $F_{p}$ is a the union of two rational curves meeting transversally.
		\item If $P \in \Sing(C)$ then $F_{p}$ is a non-reduced rational curve.
	\end{enumerate}
	Further if $H$ is a smooth curve meeting $C$ transversely away from $\Sing(C)$ then $f^{*}H$ is smooth.
\end{lemma}
\begin{proof}
	This is essentially \cite[Proposition 1.8]{sarkisov1983conic}. We sketch the proof as our statement is slightly different.
	
	Since $X$ is smooth $-K_{X}$ is relatively ample and defines an embedding into a $\mathbb{P}^{2}$ bundle over $Z$.
	Fix any point $P$ in $X$ then in some neighbourhood $U$ around $P$, $X_{U}$ is given inside $\mathbb{P}^{2} \times U$ by the vanishing of $x^{t}Qx$. Here $Q$ is a diagonalisable $3\times 3$ matrix taking coefficients in $\kappa[U]$, unique up to invertible linear transformation, so we may take $C$ to be the divisor on which the rank of $Q$ is less than $3$. That $Q$ has rank $3$ on some open set follows from smoothness of the generic fibre.
	
	Then the singular points of $C$ are precisely the locus on which $Q$ has rank less than $2$. By taking a diagonalisation of $Q$ we may write $X_{U}$ as the vanishing of $\sum A_{i}x_{i}^{2}$ for some $A_{i} \in \kappa[U]$ and we obtain the classification of fibres by consideration of the rank.
	
	Suppose then $H$ is a smooth curve as given. Away from $C$, $f^{*}H$ is clearly smooth, so it suffices to consider the intersection with $C$, however we can see it is smooth here by computing the Jacobian using the local description of $X$ given above. 
\end{proof}

	\begin{theorem}[Embedded resolution of surface singularities]\cite[Theorem 1.2]{cutkosky2009resolution}
	Suppose that $V$ is a non-singular threefold, $S$ a reduced surface in $V$ and $E$ a simple normal crossings divisor on $V$ then there is a sequence of blowups $\pi\colon V_{n} \to V_{n-1} \to ... V$ such that the strict transform $S_{n}$ of $S$ to $V_{n}$ is smooth. Further each blowup is the blowup of a non-singular curve or a point and the blown up subvariety is contained in the locus of $V_{i}$ on which the preimage of $S+E$ is not log smooth.
\end{theorem}

\begin{corollary}\label{c_res_1}
	Suppose $(X,\Delta) \to Z$ is a regular, tame conic bundle and we fix a very ample linear system $|A|$ on $Z$.  Then there is a log resolution $(X',\Delta') \to (X,\Delta)$ such that for any sufficiently general element $H\in |A|$, its pullback $G'$ to $X'$ has $(X',G'+E)$ log smooth for $E$ the reduced exceptional divisor of $\pi$.
\end{corollary}
\begin{proof}
	
	By the previous theorem we may find birational morphism $\pi\colon X' \to X$ which is a log resolution of $(X,\Delta)$ factoring as blowups $X'=X_{n} \to X_{n-1} \to .... X_{0}=X$ of smooth subvarieties contained in the non-log smooth locus of each step.
	
	We show first a general $G'$ is smooth. At each stage we blow-up smooth curves $V_{i}$ in the non-log smooth locus. Let $G_{i}$ be the pullback of $H$ to $X_{i}$, suppose for induction it is smooth. That $G_{0}$ is smooth is the content of \autoref{disc-smooth} and so the base case of the induction argument holds.
	
	We may assume that $f_{i,*}V_{i}=V_{Z,i}$ is a curve for $f_{i}\colon X_{i} \to X \to Z$ else a general $H$ avoids it and so a general $G_{i+1}$ is smooth also. Note that each vertical component of $\Delta$ is log smooth near the generic point of their image, since $X$ is a regular conic bundle, so $V_{i}$ must be contained in the strict transform of some horizontal component of $\Delta$. Since $V_{i}$ is not contracted, it follows that $V_{i} \to V_{Z,i}$ is separable as $(X,\Delta,Z)$ is tame. Thus as a general $H$ meets $V_{Z,i}$ transversely, a general $G_{i}$ meets $V_{i}$ transversely and hence a general $G_{i+1}$ is smooth. By induction then $G'=G_{n}$ is smooth.
	
%
	Suppose that $V$ is a curve contained in the locus on which $\pi^{-1}$ is not an isomorphism that is not contracted by $f$. Then for a general point $P$ of $V$, we claim that the fibre over $P$ is log smooth. As before we argue by induction, the the base case trivially true. Suppose then that we blowup a curve $V_{i}$ lying over $V$ on $X$ and $V_{Z}$ on $Z$. Then $V_{i}$ must meet the fibre over $P$ transversally. Indeed $V_{i} \to V \to V_{Z}$ is separable, as above, forcing $V_{i} \to V$ to be separable also. But then $V_{i}$ meets a general fibre transversally as claimed.
	
	Suppose now that $E$ is an integral exceptional divisor of $X'\to X$. Let $V=\pi_{*}E$, then as before general $G$ meets $V$ transversely if $V$ is a curve, or not at all otherwise. Suppose $V$ is a curve, then for a general point $P$ of $V$, the fibre over $P$ is a system of log smooth curves. Finally then the intersection of a general $G'$ and $E$ is a scheme of pure dimension $1$ contained in the disjoint union of such systems of log smooth curves, in particular it is log smooth. 
	
	Suppose then we fix two exceptional divisors $E_{1},E_{2}$ meeting at a curve $V$. Again we suppose that $V$ is not contracted by $f'=f \circ \pi$. Write $\pi_{*}V=V_{X}$ and $f'_{*}V=V_{Z}$. Then $V_{X} \to V_{Z}$ is separable as before and for a general $G'$ meeting $V$ transversely, the intersection of $G$ with $\pi^{*}V'$ is a log smooth system of rational curves, and then $G.V \subseteq G.\pi^{*}V_{X}$ is log smooth, or equally it is finitely many points with multiplicity $1$. 
\end{proof}

	\begin{theorem}
	Let $(X,\Delta) \to Z$ be a regular, tame conic bundle and $|A|$ a very ample linear system on $Z$. Then there is a log resolution $(X',\Delta') \to (X,\Delta)$ such that for a general $H \in |A|$, the pullback $G'$ to $X'$ is smooth with $(X',\Delta'+G')$ log smooth. 
\end{theorem}
\begin{proof}
	Write $E$ for the reduced exceptional divisor. For a general $H \in |A|$ we let $G=f^{*}H$ be the pullback to $X$. We then take $X'$ as in \autoref{c_res_1}.
	
	Clearly a general $G'$ avoids the intersection of any $3$ components of $\text{Supp}(\Delta')+E$, and from above $(X',G'+E)$ is log smooth. Suppose $D$ is a vertical component of $\Delta$. Then either $G$ can be assumed to avoid it, or to meet it at a smooth fibre. By the usual arguments, since the only non-contracted curves we blow up map separably onto their image, $G'$ meets $D'$ the strict transform of $D$ on $X'$ along a log smooth locus. Further this locus meets any exceptional divisor either transversally or not at all. Now suppose $D_{2}$ is any other component of $\text{Supp}(\Delta')+E$ which does not dominate $Z$. Then if either $D_{2}.D'$ has dimension less than $1$ or is contracted over $Z$ then a general $G'$ avoids it, so suppose otherwise. In which case $D_{2}$ must be exceptional over $X$ with image $V\subseteq D$ on $X$. However $D_{2}.D'$ is just the strict transform of $V$ inside $D'$ and, for a general $G'$, $G'.D_{2}.D$ is log smooth as required. 
	
	It remains then to consider the horizontal components of $\Delta$. Let $D$ be any such component and $D'$ its strict transform. Since $(X,\Delta,Z)$ is tame, so is $(X',\Delta',Z)$. In particular then $D' \to Z$ is divisorially tamely ramified and so residually separated over $Z$ away from finitely many points of $Z$. Hence by Bertini's Theorem, \autoref{Bertini}, the pullback of a general $H$, which is just the intersection of a general $G'$ with $D'$ is smooth. Further as $D' \to Z$ is divisorially tamely ramified, if $V$ is any curve on $D'$ not contracted over $Z$ a general $G'|_{D'}$ meets it transversally. Hence for any other component $D_{2}$ of $\text{Supp}(\Delta')+E$ we have $(X',D'+D_{2}+G')$ log smooth for a general $G'$ and the result follows.
\end{proof}

\begin{corollary}\label{tameAdjunction}
	Suppose $(X,\Delta,Z)$ is a terminal, sub $\epsilon$-klt, tame conic bundle. Take a general very ample $H$ on $Z$, with $G=f^{*}H$, then
	 $(G,\Delta|_{G}=\Delta_{G})$ is sub $\epsilon$-klt.
\end{corollary}
\begin{proof}
	Throwing away finitely many points of $Z$ we may freely suppose that the conic bundle is regular.
	
	By the previous theorem there is a log resolution $\pi\colon (X',\Delta') \to (X,\Delta)$ with $(X',\Delta'+G')$ smooth. Write $\pi_{G}\colon G' \to G$ for the restricted map. Then $(K_{X'}+\Delta'+G')|_{G'}=\pi_{G}^{*}(K_{G}+\Delta_{G})=K_{G'}+\Delta'|_{G}$. However $\Delta'|_{G}$ is log smooth with coefficients less than $1-\epsilon$ by construction, and hence $(G,\Delta_{G})$ is $\epsilon$-klt by assumption. 
\end{proof}

\section{$F$-Split Mori Fibre Spaces}\label{S-MFS}

The aim of this section is to prove the following theorem. 

\begin{theorem}\label{setup}
	For a field $\kappa$ of positive characteristic we let $S_{\kappa}$ be the set of $(X,\Delta)$, $\epsilon$-LCY threefold pairs with $X$ terminal, globally $F$-split and rationally chain connected over $\kappa$. We further require that $(X,\Delta)$ admits a $K_{X}$ Mori fibration $f\colon (X,\Delta) \to Z$ where either
	\begin{enumerate}
		\item $Z$ is a smooth rational curve, there is $H$ on $Z$ very ample of degree $1$ and a general fibre $G$ of $X \to Z$ is smooth.
		\[\text{or}\]
		\item $p>2$ and $(X,\Delta) \to Z$ is a tame, terminal conic bundle such that there is a very ample linear system $|A|$ on $Z$ with $A^{2} \leq c$. In which case $G$ the pullback of a sufficiently general $H \in |A|$ is smooth with $(G,\Delta_{G}=\Delta|_{G})$ $\epsilon$-klt by \autoref{tameAdjunction}.
	\end{enumerate}
	Then the set of base varieties $$S'=\{X \text{ such that } \exists \Delta \text{ with } (X,\Delta) \in S_{\kappa} \text{ for algebraically closed } \kappa\}$$ is birationally bounded over $\mathbb{Z}$. 
\end{theorem}

\begin{remark}
	In practice this will be applied to pairs over fields of characteristic $p > 7,\frac{2}{\delta}$ with boundary coefficients bounded below by $\delta$. The constraints on $p$ come from \autoref{smoothness} and \autoref{cbf}, via \autoref{S2}.
\end{remark}

This chapter is devoted to the proof, but the outline is as follows. We fix a general, very ample divisor $H$ on the base and write $G=f^{*}H$. Then argue that $A=-mK_{X}+nG$ is ample, for $m,n$ not depending on $X,\Delta$ or $G$. This is done by bounding the intersection of $K_{X}$ with curves not contracted by $f$ and generating an extremal ray in the cone of curves. We then show that in fact we may choose these $m,n$ such that $A$ defines a birational map, by lifting sections from $G$ using appropriate boundedness results in lower dimensions. The $F$-split assumption is used to lift sections from $G$ with \autoref{vanish}, it will also be needed to apply \autoref{setup} by ensuring that the bases $Z$ are suitably bounded. 

If, for some $t>0$, the non-klt locus of $(X,(1+t)\Delta)$ is contracted then since $(K_{X}+(1+t)\Delta) \sim -tK_{X}$ it follows that every $-K_{X}$ negative extremal ray is generated by a curve $\gamma$ with $K_{X}.\gamma \leq \frac{6}{t}$. In particular as we have $G.C \geq 1$ for any $-K_{X}$ negative curve $C$ it must be that $-K_{X}+\frac{7}{t}G$ is ample. Clearly for any $(X,\Delta) \to Z$ there is such a $t$, however we wish to find one independent of the pair. For this we may use a result due to Jiang, the original proof is a-priori for characteristic $0$, but the proof is arithmetic in nature and holds in arbitrary characteristic.

\begin{theorem}\label{Jiang}\cite[Theorem 5.1]{jiang2018birational}
	Fix a positive integer $m$ and $\epsilon >0$ a real number. Then there is some $\lambda$ depending only on $m,\epsilon$ satisfying the following property.
	
	Take $(T,B)$ any smooth, projective $\epsilon$-klt surface. Write $B=\sum b_{i}B_{i}$ and suppose $K_{T}+B \equiv N-A$ for $N$ nef and $A$ ample. If $B.N,\sum b_{i}, B^{2} \leq m$ then $(T,(1+\lambda)B)$ is klt. 
\end{theorem}
First we show that results of this form lift to characterisations of the non-klt locus of $(X,(1+t)\Delta)$, then show how the result above may be applied here.
\begin{lemma}
	 We use the notation of \autoref{setup}. Suppose $Z$ is a surface and there is $t$ such that $(G,(1+t)\Delta_{G})$ is klt. Then every curve in the non-klt locus of $(X,(1+t)\Delta)$ is contracted by $f$.
\end{lemma}
\begin{proof}
	Let $\pi\colon X' \to X$ be a log resolution of $(X,\Delta+G)$ with $K_{X'}+\Delta'=\pi^{*}(K_{X}+\Delta)$, then $(X',\Delta'+G')$ is log smooth and $\Delta'$ and $G$ have no common components, where $G'$ is the pullback of $G$. Now $X' \to X$ must also be a log resolution of $(X,(1+t)\Delta)$, and hence if we write $K_{X'}+B=\pi^{*}(K_{X}+(1+t)\Delta)$ then it is also true that $(X',B+G')$ is log smooth and that $B$ and $G'$ have no common components. Hence $(G',B|_{G'})$ is sub klt by assumption and in particular it has coefficients strictly less than $1$. 
	
	Suppose $Z$ is a non-klt center of $(X,(1+t)\Delta)$ and $E$ is a prime divisor lying over $Z$ inside $X'$. Then $E$ has coefficients strictly larger than $1$ in $B$. Since $(X',B+G')$ is log smooth, it must be that $E|_{G'}$ is an integral divisor and it is trivial if and only if $E$ and $G'$ do not meet. But then $E|_{G'} =\lfloor E|_{G'} \rfloor =0$ and so $E$ does not meet $G'$. Hence neither does $H$ meet $f_{*}\pi_{*}E=f_{*}Z$. In particular if $C$ is a curve in the non-klt locus, then there is an ample divisor $H$ on $Z$ not meeting $f_{*}C$. This is possible only if $f_{*}C$ is a point. 
\end{proof}

\begin{lemma}
	Using the notation of \autoref{setup} suppose that $Z$ is a curve and write $Y$ for the generic fibre of $f\colon X\to Z$. If there is $t$ such that $(Y,(1+t)\Delta_{Y})$ is klt, then every curve in the non-klt locus of $(X,(1+t)\Delta)$ is contracted by $f$.
\end{lemma}

\begin{proof}
	This follows essentially as above.
	Take a log resolution $\pi\colon (X',\Delta') \to (X,\Delta)$. Write $Y'$ for the generic fibre of $X' \to Z$. Then $(Y',\Delta'|_{Y'}) \to (Y,\Delta_{Y})$ is a log resolution. Again write $K_{X'}+B=\pi^{*}(K_{X}+(1+t)\Delta)$. Then again if $B$ has a component $D$ with coefficient at least $1$ then $D$ cannot dominate $Z$, else it would pull back to $G'$ to give a contradiction. Hence the non-klt locus of $(X,(1+t)\Delta)$ must be contracted as claimed. 
\end{proof}

\begin{lemma}
	Using the notation of the previous lemmas. There is some $\lambda$ independent of $(X,\Delta)$ and $G$ for which the non-klt locus of $(X,(1+t)\Delta)$ is contracted for all $t \leq \lambda$.
\end{lemma}
\begin{proof}
	We consider two cases. 
	
	Suppose first $Z$ is a curve, so the generic fibre $Y$ is a regular del Pezzo surface and $(G,\Delta_{G})$ is $\epsilon$-klt LCY. Then, by the work of Tanaka \cite[Corollary 4.8]{tanaka2019boundedness}, $(-K_{G})^{2} \leq 9$. We write $\Delta_{G}=\sum \lambda_{i}D_{i}$ and since $G$ is regular we have $D_{i}.K_{G} \geq 1$. Hence $\sum \lambda_{i} \leq \Delta_{G}.(-K_{G}) \leq 9$ and $\Delta_{G}^{2} =(-K_{G})^{2} \leq 9$. We conclude the result holds by \autoref{Jiang} with $N=-K_{G}$ and $A=-K_{G}$. 

	Suppose then that $Z$ is a surface. Then by \autoref{disc-smooth}, $G$ is a smooth surface, geometrically ruled over a general very ample divisor $H$ on $Z$. Further by \autoref{tameAdjunction}, $(G,\Delta_{G})$ is $\epsilon$-klt and by assumption $K_{G}+\Delta_{G}\sim kF$ where $F$ is the general fibre over $H$ and $H^{2}=k \leq c$. Finally note that $\Delta_{G}^{V} \sim_{f,\mathbb{Q}} 0$.

	We may write $\Delta_{G}= \sum \lambda_{i}D_{i}+ \sum \mu_{i}F_{i}$ where $F_{i}$ are fibres over $H$ and $D_{i}$ dominate $H$. Since $F_{i}$ is a fibre and $G$ is smooth, each $F_{i}$ is reduced by the genus formula and contains at most $2$ components since $-K_{X}.F_{i}=-2$. Further $\Delta_{G}.F=(-K_{G}).F=2$ and hence $\Delta_{G}^{2}=(-K_{G}+kF)^{2}=(-K_{G})^{2}-2kK_{G}.F +(kF)^{2} \leq (-K_{G})^{2}+4c$ which in turn is bounded above by $8+4c$ due to \cite[Proposition 11.19]{buadescu2001algebraic}, since $G$ is a smooth geometrically ruled surface. 
	
	It remains then to show that the sum of the coefficients of $\Delta_{G}$ is bounded. Note that $\sum \lambda_{i} \leq \sum \lambda_{i}D_{i}.F =\Delta_{G}.F =2$. We therefore need only bound $\sum \mu_{i}$.
	
	Suppose for contradiction that $w=\sum \mu_{i} >3 +k$. Let $B=\sum \lambda_{i}D_{i} +(1-\frac{3+k}{w})\sum \mu_{i}F_{i}  \sim -K_{G}-(F^{1}+F^{2}+F^{3})$, for general fibres $F^{i}$.
	
	Then $(G,B)$ is klt and so by \autoref{cc}, $D=F^{1}+F^{2}+F^{3}$ has 2 connected components, a clear contradiction.
	
	Therefore we may choose $A$ small and ample with $A.\Delta_{G} < c$ and write $N=kF+A$ to satisfy the conditions of \autoref{Jiang}. The result then follows as $\Delta_{G}.N=kF.\Delta_{G}+A.B\leq 3c$ is still bounded.
\end{proof}
	\begin{corollary}\label{nAmple}
	There is some $n$ such that for any $(X,\Delta) \to Z$ and $G$ as in \autoref{setup} we have $-K_{X}+nG$ is ample.
\end{corollary}

\begin{proof}
	Take any $n \geq \frac{7}{\lambda}$ for $\lambda$ as in the previous lemma. Suppose $R$ is a $K_{X}+(1+\lambda)\Delta\equiv -\lambda K_{X}$ negative extremal ray. By construction, every curve in the nlc locus is contracted by $X \to Z$ so $-\lambda K_{X}$ is positive on $\overline{NE(X)}_{nlc}\setminus \{0\}$. Hence by \autoref{NLCT} any such ray is spanned by a curve $C$ with $0< -\lambda K_{X}.C \leq 6$ and $G.C >0$. Since $G$ is Cartier, we have $G.C > 0$ and hence $(-K_{X}+nG).G \geq \frac{1}{\lambda} >0$. In particular $-K_{X}+nG$ is ample as claimed. 
\end{proof}

\begin{theorem}\label{t-ample}
	Let $(X,\Delta) \to Z$ and $G$ be as in \autoref{setup}. Then there is $t$ not depending on the pair $(X,\Delta)$ nor on $G$ with $-3K_{X}+tG$ ample and defining a birational map. 
\end{theorem}
\begin{proof}
	Consider first the case that $\dim Z=1$. Then $G$ is a smooth del Pezzo surface, so $-3K_{G}$ is globally generated by \cite[Proposition 2.14]{bernasconi2022pezzo}. Let $G_{1},G_{2}$ be other general fibres and consider
	\[0 \to \ox(-3K_{X}+kG-G_{1}-G_{2}) \to \ox(-3K_{X}+kG) \to \mathcal{O}_{G_{1}}(-3K_{G_{1}})\oplus \mathcal{O}_{G_{2}}(-3K_{G_{2}}) \to 0.\]
	
	Since $X$ is globally $F$-split $H^{i}(X,A)=0$ for all $i>0$ and $A$ ample by \autoref{vanish}. In particular then $H^{1}(X,\ox(-3K_{X}+kG-G_{1}-G_{2}))$ vanishes when $k\geq 3n+2$ for $n$ as given by the proceeding corollary. Therefore we may lift sections of $-3K_{G_{i}}$ to see that $-3K_{X}+kG$ defines a birational map for any $k \geq 3n+2$. 

	Suppose instead that $\dim Z=2$, so $G$ is a conic bundle. Choose a general $H'\sim H$ on $Z$ and let $G'$ be its pullback. Consider $A_{k}=(-K_{X}+kG)|_{G'}=(-k_{G'}+(k-1)dF)$ for $d \geq 1$, where $F$ is the general fibre of $G'\to H'$. Then $A_{k}$ is ample for $k >n$ and is Cartier since $G$ is smooth. In particular by the Fujita conjecture for smooth surfaces \cite[Corollary 2.5]{terakawa1999d}, $K_{G'}+4A_{k}$ is very ample. Choosing suitable $k,k'$ we may write $K_{G'}+4A_{k}=-3K_{G'}+4(k-1)dF=(-3K_{X}+k'G)|_{G'}$. Consider now
	\[0 \to \ox(-3K_{X}+(k'-1)G)\to \ox(-3K_{X}+k'G)\to \mathcal{O}_{G'}(-3K_{G'}+4(k-1)dF) \to 0.\]
	Again the higher cohomology of $-3K_{X}+(k'-1)G$ vanishes and we may lift sections to $H^{0}(X,\ox(-3K_{X}+k'G))$ from general fibres. In particular $-3K_{X}+k'G$ separates points on a general $G'$ so $-3K_{X}+(k'+1)G$ separates general points and thus defines a birational map. 

	We may then pick some suitably large $t$ for which the result holds as $k,k'$ were chosen independently of $(X,\Delta) \to Z$ and $G,G_{1},G_{2}$.
\end{proof}

\begin{lemma}
	Let $(X,\Delta) \to Z,S$ and $G$ be as in \autoref{setup} and $t$ as in \autoref{t-ample}. Then there is some constant $C$ with $(-3K_{X}+tG)^{3} \leq C$ and $(X,\Delta)\in S$.
\end{lemma}

\begin{proof}
	The anticanonical volumes $\Vol(X,-K_{X})$ are bounded by some $V$ by \autoref{Main2} which is proved in the next section.
	
	Suppose first $\dim Z=1$. Then $\Vol(G,-K_{G})=(-K_{G})^{2} \leq 9$ and so by Lemma 3.4
	\[\Vol(X,-3K_{X}+nG) \leq \Vol(X,-3K_{X}) + 3t\Vol(G,-3K_{G})\leq 27(V+9t)\]
	as required.

	Suppose instead then that $\dim Z=2$. So $G$ is a conic bundle over some $H$ on $Z$ with $H^{2} \leq c$. Hence we get
	\[\Vol(G,(-3K_{X}+tG)|_{G})= (-3K_{G}+(t+1)H^{2}F)^{2}=9K_{G}^{2}-2(t+1)H^{2}(K_{G}.F)\]
	where $F$ is a general fibre of $G \to H$. Hence $F$ is a smooth rational curve and $K_{G}.F=-2$ and $\Vol(G,(-3K_{X}+tG)|_{G})\leq 72+4(t+1)c$. Then as before we may apply \autoref{vol} to get 
	\[\Vol(X,-3K_{X}+tG) \leq \Vol(X,-3K_{X}) + 3n\Vol(G,(-3K_{X}+tG)|_{G})\]
	and boundedness follows.
\end{proof}

\begin{proof}[Proof of \autoref{setup}]
	Suppose $(X,\Delta) \in S$. Then $A=-3K_{X}+tG$ is birational with bounded volume by the preceding results. Thus $S'$ is birationally bounded by \autoref{l_birationally-bounded}.
	\end{proof}

\section{Weak BAB for Mori Fibre Spaces}\label{S-BAB}

This section is devoted to providing a bound on the volume of $-K_{X}$ under suitable conditions. Namely we show that the claim holds if $X$ belongs to a suitable family of $\epsilon$-LCY Mori fibre spaces whose bases are bounded. We will work over fields of characteristic $p>5$ as we will need to appeal to \autoref{WCL} at several points. In practice these results will be applied under the hypotheses of \autoref{Main2} with the constraints on characteristic needed to ensure $X$ is a tame conic bundle, or a generically smooth del Pezzo fibration as appropriate.
We consider first the case that $X$ is a tame conic bundle over a surface. 

\begin{theorem}\label{J1}
	Pick $\epsilon,c >0$. Then there is $V(\epsilon,c)$ such that if $f\colon (X,\Delta) \to S$ is any projective, tame conic bundle over any closed field of characteristic $p> 5$, $(X,\Delta)$ is $\epsilon$-klt and $S$ admits a very ample divisor $H$ with $H^{2} \leq c$, then $\Vol(-K_{X}) \leq V(\epsilon,c)$. 
\end{theorem}

We may further assume that $H$ and $G=f^{*}H$ are smooth. Moreover $H$ may be taken so that $(G,\Delta|_{G})$ is $\epsilon$-klt also by \autoref{tameAdjunction}.

If $\Vol(-K_{X})=0$ the result is trivially true, so we may suppose that $-K_{X}$ is big. In particular we may write $-K_{X}\sim A+E$ where $A$ is ample and $E \geq 0$. Note that $$-K_{X}-(1-\delta)\Delta\sim -\delta K_{X} \sim \delta A + \delta E$$ for any $0 < \delta <1$. Choose $\delta$ such that $(X,(1-\delta)\Delta+\delta E)$ and $(G,(1-\delta)\Delta|_{G}+\delta E|_{G})$ are $\epsilon$-klt and write $B=(1-\delta)\Delta+\delta E$. Then $(X,B)$ is $\epsilon$-log Fano by construction. The proof follows essentially as in characteristic zero, which can be found in \cite{jiang2014boundedness}, but we include a full proof for completeness as some details are modified.

\begin{lemma}\cite[Lemma 6.5]{jiang2014boundedness}
	With notation as above, $\Vol(-K_{X}|_{G}) \leq \frac{8(c+2)}{\epsilon}$. 
\end{lemma}
\begin{proof}
	Suppose for contradiction $\Vol(-K_{X}|_{G}) >\frac{8(c+2)}{\epsilon}$ and choose $r$ rational with $\Vol(-K_{X}|_{G}) > 4r >\frac{8(c+2)}{\epsilon}$.
	
	Write $F$ for the general fibre of $G \to H$. Then $G|_{G}=H^{2}F=kF$ and for suitably divisible $m$ and any $n$ we have the following short exact sequence.
	
	\[0 \to \mathcal{O}_{G}(-mK_{X}|_{G}-nF) \to \mathcal{O}_{G}(-mK_{X}|_{G}-(n-1)F) \to \mathcal{O}_{F}(-mK_{F}) \to 0\]
	
	In particular then $h^{0}(G,-mK_{X}|_{G}-nF) \geq h^{0}(G,-mK_{X}|_{G}-(n-1)F)-h^{0}(F,-mK_{F})$.
	Hence by induction we have $h^{0}(G,-mK_{X}|_{G}-nF) \geq h^{0}(G,-mK_{X}|_{G})-n\cdot h^{0}(F,-mK_{F})$.
	
	Note however that, letting $n=mr$ we have $$\lim_{m \to \infty} \frac{2}{m^{2}}(h^{0}(G,-mK_{X}|_{G})-n\cdot h^{0}(F,-mK_{F}))= \Vol(-K_{X}|_{G})-2r\Vol(-K_{F}) > 0$$ since $F$ is a smooth rational curve. Hence $-mK_{X}|_{G}-mrF$ admits a section for $m$ sufficiently large and divisible. Choose an effective $D\sim_{\mathbb{Q}} -K_{X}|_{G}-rF$.
	
	Consider now \[(G,(1-\frac{k+2}{r})B|_{G}+\frac{k+2}{r}D+F_{1}+F_{2})\]
	 for two general fibres $F_{1}, F_{2}$.
	This has \begin{align*}
	&-K_{G}+(1-\frac{k+2}{r})B|_{G}+\frac{k+2}{r}D+F_{1}+F_{2}	\\
	\sim & -(K_{X}|_{G}+kF)+\frac{k+2}{r})B|_{G}+\frac{k+2}{r}(-K_{X}|_{G}-rF)+F_{1}+F_{2} \\	
	\sim & -(1-\frac{k+2}{r})(K_{X}+B)|_{G} \\
	\end{align*}
	and hence we may apply the Connectedness Lemma for surfaces, \autoref{Tcl}, to see that its non-klt locus is connected. Note that we have $r> c+2 \geq k+2$ and so as $-(K_{X}+B)$ is ample, this pair satisfies the assumptions of the Connectedness Lemma.
	
	Since both $F_{1}$ and $F_{2}$ are contained in the non-klt locus, there must be a non-klt center $W$ dominating $H$. Thus it follows that $(F,(1-\frac{k+2}{r})B|_{F}+\frac{k+2}{r}D|_{F})$ is non-klt. However $(F,(1-\frac{k+2}{r})B|_{F})$ is $\epsilon$-klt so we must have $\deg (\frac{k+2}{r}D|_{F})\geq \epsilon$. Finally since $D|_{F}\sim -K_{X}|_{F}=-K_{F}$ we have $\deg(D|_{F})=2$ and hence $\frac{2(c+2)}{r} \geq \frac{2(k+2)}{r} \geq \epsilon$, contradicting the choice of $r$.
\end{proof}

\begin{proof}[Proof of \autoref{J1}]
	
	Take $V(\epsilon,c)=\frac{144(c+2)}{\epsilon^{2}}$ suppose for contradiction that $ \Vol(-K_{X}) > \frac{144(c+2)}{\epsilon^{2}}$. Choose $t$ with $\Vol(-K_{X})> t\cdot\frac{24(c+2)}{\epsilon} > \frac{144(c+2)}{\epsilon^{2}}$ and consider the following short exact sequence.
	\[0 \to \ox(-mK_{X}-nG)\to \ox(-mK_{X}+(n-1)G)\to \mathcal{O}_{G}(-mK_{X}|_{G}-(n-1)G) \to 0\]
	
	Arguing as before we see that $h^{0}(X,-mK_{X}-tmG)$ grows like $\frac{r}{6}m^{3}$ with $r\geq \Vol(-K_{X})-3t\Vol(-K_{X}|_{G}) >0$ by the previous lemma. In particular we may find $D \sim_{\mathbb{Q}} -K_{X}+tG$.
	
	Let $\pi\colon Y \to X$ be a log resolution of $(X, (1-\frac{3}{t})B+\frac{3}{t}D)$. We may write $K_{Y}+\Delta_{Y}+E=\pi^{*}(K_{X}+(1-\frac{3}{t})B+\frac{3}{t}D)$ where $(Y,\Delta_{Y})$ is sub klt and $E$ is supported on the non-klt places of $(X, (1-\frac{3}{t})B+\frac{3}{t}D)$. 
	
	As shown by Tanaka in \cite[Theorem 1]{tanaka2017semiample}, since $|L|=\pi^{*}f^{*}|H|$ is base point free there is some $m$ with $(Y,\Delta_{Y}+\frac{1}{m}(L_{1}+L_{2}+L_{3}))$	 still klt for every choice of $L_{i} \in |L|$. In particular, fixing some general $z\in Z$ we may take $H_{i} \in |H|$ meeting $Z$ for $1\leq  i \leq 2m$ such that for any $I \subseteq \{0,1,...,2m\}$ with $|I| =3$ the following hold:
	
	\begin{itemize}
		\item $(Y,\Delta_{Y}+\sum_{i \in I}\frac{1}{m}\pi^{*}f^{*}H_{i})$ is klt;
		\item $\bigcap_{i\in I} H_{i}={z}$.
	\end{itemize} 
	
	Thus we must have 
	\[\nklt(X,(1-\frac{3}{t})B+\frac{3}{t}D)=\nklt(X, (1-\frac{3}{t})B+\frac{3}{t}D+\frac{1}{m}f^{*}H_{i})\]
	for each $i$. 
	
	Let $F$ be the fibre over $z$ and $G_{1} = \sum_{i=1}^{2m} \frac{1}{m}H_{i}$. Then clearly $\text{mult}_{F}(G_{1}) \geq 2$ and hence $(X,G)$ cannot be klt at $F$. By construction we have
	
	\[\nklt(X, (1-\frac{3}{t})B+\frac{3}{t}D) \cup F = \nklt(X, (1-\frac{3}{t})B+\frac{3}{t}D+G_{1}).\]
	
	Similarly we may further take $G_{2} \sim f^{*}H$ not containing $F$ such that
	\[\nklt(X, (1-\frac{3}{t})B+\frac{3}{t}D+G_{1}+G_{2}) = \nklt(X, (1-\frac{3}{t})B+\frac{3}{t}D+G_{1}).\]
	
	Now $-(K_{X}+(1-\frac{3}{t})B+\frac{3}{t}D+G_{1}+G_{2}) \sim (1-\frac{3}{t})(K_{X}+B)$ is ample, so we may apply the Connectedness Lemma, \autoref{WCL}, to see there is a curve in the non-klt locus of $(X,(1-\frac{3}{t})B+\frac{3}{t}D)$ meeting $F$. In particular then the non-klt locus dominates $S$. Hence we must also have that $(F,(1-\frac{3}{t})B|_{F}+\frac{3}{t}D|_{F})$ is not klt. However $(F,B|_{F})$ is $\epsilon$-klt and $F$ is a smooth rational curve. Therefore by degree considerations, since $-K_{X}|_{F} \sim D|_{F}$ we must have $t \leq \frac{6}{\epsilon}$, contradicting our choice of $t$. 
\end{proof}

\begin{theorem}[Ambro-Jiang Conjecture for surfaces]\label{aj}\cite[Theorem 2.8]{jiang2014boundedness}
	Fix $0<\epsilon<1$. There is a number $\mu(\epsilon)$ depending only on $\epsilon$ such that for any surface $S$ over any closed field $k$, if $S$ has a boundary $B$ with $(S,B)$ $\epsilon$-klt weak log Fano then \[\inf \{ulct (S,B;G) \text{ where } G \sim_{\mathbb{Q}}-(K_{S}+B) \text{ and } G+B \geq 0\}\geq \mu(\epsilon)\]
\end{theorem}

Here $ulct (S,B;G)= \sup\{t\colon  (S,B+tG) \text{ is lc and } 0 \leq t \leq 1\}$ and in particular it is at most the usual lct, if $G$ is effective.

Though the proof is given for characteristic zero, it is essentially an arithmetic proof that the result holds for $\mathbb{P}^{2}$ and $\mathbb{F}_{n}$ for $n \leq \frac{2}{\epsilon}$. The arguments of the proof work over any algebraically closed field and as the bound is given explicitly in terms of $\epsilon$ it is independent of the base field.

By applying this result to a general fibre of a Mori fibration over a curve we obtain the desired boundedness result.

\begin{theorem}\label{J2}
	Pick $\epsilon>0$. Suppose that $f\colon X \to \mathbb{P}^{1}$ is a terminal threefold Mori fibre space with smooth generic fibre over a closed field of characteristic $p> 5$. If there is a pair $(X,\Delta)$ which is $\epsilon$-LCY then $\Vol(-K_{X})\leq W(\epsilon)$ for some $W(\epsilon)$ depending only on $\epsilon$.
\end{theorem}

\begin{proof}
	By \autoref{nAmple}, there is some $t(\epsilon)\geq 1$ depending only on $\epsilon$ with $-K_{X}+t(\epsilon)F$ ample, where $F$ is a general fibre.
	
	Let $\mu=\mu(1)$ as given in \autoref{aj} and take $W(\epsilon)= \frac{27(t(\epsilon)+2)}{\mu}$. Suppose for contradiction $\Vol(-K_{X}) > W(\epsilon)$ and choose $s$ rational with $\Vol(-K_{X}) > 27s > W(\epsilon)$. Clearly $s > \frac{(t(\epsilon)+2)}{\mu} > t(\epsilon)+2$. 
	
	For any $n$ and for sufficiently divisible $m$, we have the following short exact sequence.
	\[0 \to \ox (-mK_{X}-nF) \to \ox(-mK_{X}-(n-1)F) \to \mathcal{O}_{F}(-mK_{F})\to 0.\]
	This gives $h^{0}(X,-mK_{X}-nF) \geq h^{0}(X,-mK_{X})-nh^{0}(F,-mK_{F})$ and subsequently 
	\[\lim_{m \to \infty} \frac{6}{m^{3}}(h^{0}(X,-mK_{X})-smh^{0}(F,-mK_{F})= \Vol(-K_{X})-3s\Vol(-K_{F}).\] Since $F$ is a smooth del Pezzo surface we have $\Vol(-K_{F})\leq 9$.  So by construction $-mK_{X}-smF$ is effective for large, divisible $m$. 
	
	Choose $D\geq 0$ with $D \sim_{\mathbb{Q}} -K_{X}-sF$ and consider $(X,\frac{t(\epsilon)+2}{s}D+F_{1}+F_{2})$ for $F_{1},F_{2}$ general fibres. By construction we have
	\begin{align*}
	-(K_{X}+\frac{t(\epsilon)+2}{s}D+F_{1}+F_{2})&\sim -(K_{X}-\frac{t(\epsilon)+2}{s}K_{X}-t(\epsilon)F)\\
	&\sim (1-\frac{t(\epsilon)+2}{s})(-K_{X}+t(\epsilon)F) + \frac{t(\epsilon)(t(\epsilon)+2)}{s}F
	\end{align*}
	which is ample since $F$ is nef and $-K_{X}+t(\epsilon)F$ is ample.
	Then the Connectedness Lemma, \autoref{WCL}, gives that the non-klt locus is connected, and clearly contains $F_{1},F_{2}$, so it must contain a non-klt center $W$ which dominates $\mathbb{P}^{1}$. Thus it must be that $(F,\frac{t(\epsilon)+2}{s}D|_{F})$ is not klt. However $F$ is smooth, and equivalently terminal, with $-K_{F}\sim D|_{F}$ ample, so by \autoref{aj} it follows that $\frac{t(\epsilon)+2}{s} \geq lct(F,0;D|_{F}) \geq \mu=\mu(1)$. Thus we have $s \leq \frac{t(\epsilon)+2}{\mu}$ contradicting our choice of $s$ and proving the result.
\end{proof}

\section{Birational Boundedness}\label{S-res}

We are now ready to prove the main theorems using the results of the previous sections.

\begin{lemma}\label{S1}
	Suppose that $(X,\Delta)$ is an $\epsilon$-klt LCY pair in characteristic $p>5$, with $\Delta \neq 0$ and $X$ both rationally chain connected and $F$-split. Then there is a birational map $\pi\colon X \dashrightarrow X'$ such that $X'$ has a Mori fibre space structure $X' \to Z$ and $\Delta'=\pi_{*}\Delta$ on $X'$ making $(X',\Delta')$ klt and $\epsilon$-LCY. Further both $X'$ and $Z$ are rationally chain connected and $F$-split and if $X$ is terminal, so is $X'$.
\end{lemma}
\begin{proof}
	Replacing $X$ by a $\mathbb{Q}$-factorialisation, we can assume $X$ is $\mathbb{Q}$-factorial. This can be done by \autoref{dlt}.
	
	Since $(X,\Delta)$ is klt so is $(X,0)$ and hence we may run a terminating $K_{X}$ MMP $X=X_{0} \dashrightarrow X_{1} \dashrightarrow... \dashrightarrow X_{n}=X'$ by \autoref{Cone Theorem}. At each step $X_{i} \dashrightarrow X_{i+1}$ we may pushforward $\Delta_{i}$ to $\Delta_{i+1}$, which is still klt since $K_{X}+\Delta \equiv 0$. Similarly since $X_{i}$ is $F$-split and rationally chain connected, so is $X_{i+1}$ as these are preserved under birational maps of normal varieties. Since $K_{X}$ cannot be pseudo-effective, $X'$ has a Mori fibre space structure $X' \to Z$, where $Z$ is also rationally chain connected and $F$-split. If $X$ is terminal we may run a $K_{X}$ MMP terminating at a terminal variety, hence $X'$ is terminal also.	
\end{proof}

\begin{reptheorem}{Main}
		Fix $0 < \delta, \epsilon <1$. Let $S_{\delta,\epsilon}$ be the set of threefolds satisfying the following conditions:
\begin{itemize}
	\item $X$ is a projective variety over an algebraically closed field of characteristic $p >7, \frac{2}{\delta}$;
	\item $X$ is terminal, rationally chain connected and $F$-split;
	\item $(X,\Delta)$ is $\epsilon$-klt and log Calabi-Yau for some boundary $\Delta$; and
	\item The coefficients of $\Delta$ are greater than $\delta$.
\end{itemize}

Then there is a set $S'_{\delta,\epsilon}$, bounded over $\text{Spec}(\mathbb{Z})$ such that any $X\in S_{\delta,\epsilon}$ is either birational to a member of $S'_{\delta,\epsilon}$ or to some $X'\in S_{\delta,\epsilon}$, Fano with Picard number $1$.
\end{reptheorem}
\begin{proof}
	Take any $(X,\Delta)\in S$ and replace it by a Mori fibre space $(X',\Delta') \to Z$ by \autoref{S1}. Then $Z$ is $F$-split and rationally chain connected. If $Z$ is a surface then $p>\frac{2}{\delta}$ ensures that $(X',\Delta')\to Z$ is a tame conic bundle by \autoref{S2}. In particular $Z$ admits a boundary $\Delta_{Z}$ such that $(Z,\Delta_{Z})$ is $\epsilon$-LCY by \autoref{cbf}. Hence by BAB for surfaces, \autoref{SBAB}, there is $|A|$ a very ample linear system on $Z$ with $A^{2}\leq c$ for some $c$ independent of $X,\Delta,Z$.
	
	On the other hand, if $Z$ is a curve then it is a smooth rational curve and $p>7$ gives that the general fibre of $X \to Z$ is smooth by \autoref{smoothness}. Let then $S'_{\delta,\epsilon,V}$ be set of such Mori fibre space $(X',\Delta') \to Z$ with $Z$ not a point and $\Vol(-K_{X})\leq V(\epsilon,c)$. In both cases we conclude by \autoref{setup} that the set is birationally bounded.	
	\end{proof}

\begin{reptheorem}{Main2}
	Fix $0 < \delta, \epsilon <1$ and let $T_{\delta,\epsilon}$ be the set of threefold pairs $(X,\Delta)$ satisfying the following conditions
	\begin{itemize}
		\item $X$ is projective over a closed field of characteristic $p >7,\frac{2}{\delta}$;
		\item $X$ is terminal, rationally chain connected and $F$-split;
		\item $(X,\Delta)$ is $\epsilon$-klt and LCY;
		\item The coefficients of $\Delta$ are greater than $\delta$; and
		\item $X$ admits a Mori fibre space structure $X \to Z$ where $Z$ is not a point.
	\end{itemize}
	Then the set $\{\Vol(-K_{X}) \colon \exists \Delta \text{ with } (X,\Delta) \in T_{\delta,\epsilon}\}$ is bounded above. 
\end{reptheorem}

\begin{proof}
	
	Take $(X,\Delta) \in T_{\delta,\epsilon}$ and let $X \to Z$ be the associated Mori fibre space structure. If $Z$ is a curve then we conclude that $\Vol(-K_{X})$ is bounded by \autoref{J2} in light of \autoref{smoothness}. If instead $Z$ is a surface then the set of possible such $Z$ is bounded by \autoref{SBAB} and \autoref{cbf} as above. Hence we conclude the claim by \autoref{J1}.
\end{proof}

\addtocontents{toc}{\protect\setcounter{tocdepth}{-1}}
\section*{Acknowledgements}
I would like to thank Paolo Cascini for his constant, and unerringly patient, support in writing this. His help was indispensable throughout the process, from the choice of topic itself to the resolution of many issues in earlier drafts. Thanks also to the reviewer for their careful reading and helpful comments.

I would also like to thank Federico Bongiorno for answering all the questions I was too embarrassed to ask Paolo, and finally the EPSRC for my funding. 

\bibliography{BoundRef}
\bibliographystyle{amsalpha}

\end{document}